\documentclass[12pt,a4paper]{amsart}
\usepackage{algorithm}
\usepackage{algorithmic}
\usepackage{amssymb}
\usepackage{eucal}
\usepackage{graphicx}
\usepackage{amsmath}
\usepackage{amscd}
\usepackage[all]{xy}           %xypic macro for latex2.09
\usepackage{tikz}
\usepackage{tikz-qtree}
\usepackage{amsfonts,latexsym}
\usepackage{xspace}
\usepackage[T1]{fontenc}
\usepackage{tikz-cd}
\usepackage{txfonts}

\usepackage{epsfig}
\usepackage{float}
\usepackage{color}
\usepackage{fancybox}
\usepackage{colordvi}
\usepackage{multicol}
\usepackage{colordvi}
\ifpdf
\usepackage[colorlinks,final,backref=page,hyperindex]{hyperref}
\else
\usepackage[colorlinks,final,backref=page,hyperindex]{hyperref}
\fi
\usepackage[active]{srcltx} %SRC Specials for DVI Searching
\usepackage{mathrsfs} %%

\newtheorem{thm}{Theorem}[section]
\newtheorem*{theorem*}{Theorem}

\newtheorem{prop}[thm]{Proposition}
\newtheorem{lm}[thm]{Lemma}

\newtheorem{coro}[thm]{Corollary}

\makeatletter

\newcommand{\nc}{\newcommand}

\nc{\delete}[1]{{}}

    \nc{\mlabel}[1]{\label{#1}}  % Use this to suppress names
    \nc{\mcite}[1]{\cite{#1}}  % Use this to suppress names
    \nc{\mref}[1]{\ref{#1}}  % Use this to suppress names
    \nc{\meqref}[1]{\eqref{#1}}  % Use this to suppress names
    \nc{\mbibitem}[1]{\bibitem{#1}} % Use this to show number
\delete{
    \nc{\mlabel}[1]{\label{#1} {{\small\tt{{\ }\ (#1)}}}}                % Use this lines to show names
    \nc{\mcite}[1]{\cite{#1}{{\small\tt{{\ }(#1)}}}}  % Use this lines to show names
    \nc{\mref}[1]{\ref{#1}{{\small\tt{{\ }(#1)}}}}  % Use this lines to show names
    \nc{\meqref}[1]{\eqref{#1}{{\small\tt{{\ }(#1)}}}}  % Use this lines to show names
    \nc{\mbibitem}[1]{\bibitem[\bf #1]{#1}} % Use this to show name
}

\newcommand*\bigcdot{\mathpalette\bigcdot@{.5}}
\newcommand*\bigcdot@[2]{\mathbin{\vcenter{\hbox{\scalebox{#2}{$\m@th#1\bullet$}}}}}
\makeatother

\newtheorem{innercustomgeneric}{\customgenericname}
\providecommand{\customgenericname}{}
\newcommand{\newcustomtheorem}[2]{%
    \newenvironment{#1}[1]
    {%
        \renewcommand\customgenericname{#2}%
        \renewcommand\theinnercustomgeneric{##1}%
        \innercustomgeneric
    }
    {\endinnercustomgeneric}
}
\newcustomtheorem{customthm}{Theorem}

\theoremstyle{definition}
\newtheorem{example}[thm]{Example}
\newtheorem{df}[thm]{Definition}
\newtheorem{remark}[thm]{Remark}

\nc{\name}[1]{{\bf #1}}

\allowdisplaybreaks

\newcommand{\N}{\mathbb{N}}
\newcommand{\Z}{\mathbb{Z}}
\newcommand{\Q}{\mathbb{Q}}

\nc{\Res}{\mathrm{Res}}

\def \ra {\rightarrow}

\def \C {\mathbb{C}}

\def\la{\lambda}

\def \al{\alpha}

\def \om{\omega}

\def \b {\beta}
\def \op {\oplus}

\def \ssq{\subseteq}

\def \vac {\mathbf{1}}

\def \g {\mathfrak{g}}
\def \h {\mathfrak{h}}

\def \Hom {\mathrm{Hom}}

\def \End {\mathrm{End}}

\def\Id{\mathrm{Id}}
\def \wt {\mathrm{wt}}

\def \bs {\backslash}

\def\RBO{\mathrm{RBO}}
\def\o{\otimes}
\def\Diff{\mathrm{Diff}}
\def\Der{\mathrm{Der}}
\def\Aut{\mathrm{Aut}}
\def\Y{\mathcal{Y}}

\def\<{\langle}
\def\>{\rangle}

\newcommand{\lir}[1]{\textcolor{red}{\underline{Li:}#1 }}

\newcommand{\cm}[1]{\textcolor{orange}{\underline{CM:}#1 }}
\newcommand{\jq}[1]{\textcolor{blue}{\underline{JQ:}#1 }}

\newcommand{\fusion}[3]{{\binom{#3}{#1\;#2}}}

\begin{document}

\title[On Rota-Baxter vertex operator algebras]{On Rota-Baxter vertex operator algebras}

\author{Chengming Bai}
\address{Chern Institute of Mathematics \& LPMC, Nankai University, Tianjin 300071, China}
\email{baicm@nankai.edu.cn}

\author{Li Guo}
\address{Department of Mathematics and Computer Science, Rutgers University, Newark, NJ 07102, USA}
\email{liguo@rutgers.edu}

\author{Jianqi Liu}
\address{Department of Mathematics, University of California, Santa Cruz, CA 95064, USA}
\email{jliu230@ucsc.edu}

\author{Xiaoyan Wang}
\address{Department of Mathematics and Computer Science, Rutgers University, Newark, NJ 07102, USA}
\email{xw350@newark.rutgers.edu}

\date{\today}

\begin{abstract}
Derivations play a fundamental role in the definition of vertex
(operator) algebras, sometimes regarded as a generalization of
differential commutative algebras. This paper studies the role
played by the integral counter part of the derivations, namely
Rota-Baxter operators, in vertex (operator) algebras. The closely
related notion of dendriform algebras is also defined for vertex
operator algebras. It is shown that the classical relations among
dendriform algebras, associative algebras and Rota-Baxter algebras are preserved for their vertex algebra analogs.
\end{abstract}

\subjclass[2010]{
    17B69,  % Vertex operators; vertex operator algebras and related structures
    17B38, %Yang-Baxter equations and Rota-Baxter operators
    17B10, %Representations of Lie algebras
    81R10  %Infinite-dimensional groups and algebras motivated by physics, including Virasoro, Kac-Moody, W-algebras and other current algebras and their representations
    17B68, %Virasoro and related algebras
    17B65,  %Infinite-dimensional Lie (super)algebras
    81R12  %Groups and algebras in quantum theory and relations with integrable systems
    %81T40, % Two-dimensional field theories, conformal field theories, etc. in quantum mechanics
    %81R15 %Operator algebra methods applied to problems in quantum theory
}

\keywords{Vertex algebra, vertex operator algebra, derivation,
Rota-Baxter algebra, field algebra, dendriform algebra}

\maketitle

\vspace{-1.2cm}

\tableofcontents

\vspace{-1.2cm}

\allowdisplaybreaks

\section{Introduction}
This paper studies Rota-Baxter operators on vertex operator algebras (VOAs), as an integral counterpart of the derivations which play an essential role in VOAs. The closely related dendriform algebras for VOAs are also studied.

\vspace{-.3cm}
\subsection{Vertex operator algebras and derivations}
Vertex algebras and VOAs, introduced by Borcherds in \mcite{Bor} and Frenkel, Lepowsky and Meurman in \mcite{FLM} respectively, were developed in conjunction with conformal field theory, ``monstrous moonshine'', string theory and infinite-dimensional Lie algebras, with applications in geometric Langlands program. See \mcite{FLM,Ka, FB-Z, FZ,TUY}.

On the other hand, the study of differential algebras, defined to
be an associative algebra equipped with a derivation, has its
origin in Ritt's algebraic approach to differential
equations~\mcite{Ri} and further developed by Kolchin~\mcite{Ko}.
Since then, the subject has evolved into a vast area involving
differential aspects of Galois theory, algebraic geometry,
computational algebra and logic (see for example~\mcite{PS}). More
generally, differential algebras with weights were introduced as
an abstraction of the differential quotients whose limits define
the derivation in analysis~\mcite{GK}.

%A commutative unital differential algebra naturally gives rise to a vertex algebra. 

A vertex algebra in general can be viewed as a
generalization of a commutative differential algebra. In fact, any
commutative differential unital algebra $(A,D,1)$ is naturally a
vertex algebra (cf. \mcite{Bor}). More generally, the notion of vertex
algebras, as well
as the related notion of vertex algebras without vacuum (cf. \mcite{HL}), can be equivalently formulated
in terms of a ``weakly commutative vertex operator'' equipped with a special
derivation. See~\mcite{Ka,HL,LL,LTW} and Theorem~\mref{thm2.3}
and Proposition~\mref{prop2.6}. On the other hand, the study of derivations of vertex (operator) algebras has drawn much attention. The derivations of VOAs naturally give rise to automorphisms of VOAs (cf. \mcite{FHL,DG1}), which is an important notion that relates to many subbranches like the moonshine conjecture, orbifold theory, and quantum Galois theory, etc., see \mcite{FLM,DRX,DM}. Huang observed in \mcite{H1} that the set of derivations of a VOA with coefficients in an ordinary module can be identified with the first-order cohomology of VOAs, similar to the case of Lie algebras and associative algebras. The structure of the derivation algebras of strongly rational VOAs was determined by Dong and Griess in \mcite{DG1}. They also proved in \mcite{DG1} that the derivation algebra of a strongly rational VOA $V$ generates the connected component of the automorphism group of $V$. 

%The automorphism group of the moonshine module VOA $V^\natural$ is the monster simple group $\mathcal{M}$ (cf. \mcite{FLM}),

\vspace{-.2cm}
\subsection{Vertex operator algebras and Rota-Baxter operators}
Given the close relationship between derivations and vertex (operator)
algebras, it is natural to study the role played by integrations
in vertex (operator) algebras. The integral counterpart of the differential
algebra (with a weight) is the Rota-Baxter algebra, whose study originated from the 1960 work of
G.~Baxter~\mcite{Ba} in probability and can be tracked back
further under the disguise of a linear
transformation~\mcite{Co,Tri}. For Lie algebras, they were
rediscovered in the 1980s as the operator forms of the classical
Yang-Baxter equation (CYBE) and modified Yang-Baxter
equation. To further develop the operator forms of the
CYBE~\mcite{S}, the more general notion of a relative Rota-Baxter
operator was introduced by Kupershmidt who called it an
$\mathcal{O}$-operator.~\mcite{Ku}

Rota-Baxter operators have been defined on a very wide range of algebraic structures, and indeed on algebraic operads~\mcite{PBG}. Given the importance of VOAs and Rota-Baxter algebras, it is a natural question to investigate the possibility of combining the theories of Rota-Baxter operators and VOAs. However, due to the complexity of VOAs, this has not been carried out. 

%\jq{I made a slight modification of the wording here.}

The purpose of this paper is to define, on VOAs, the closely related notions of differential operators, Rota-Baxter operators and dendriform algebras. As it turns out, in
a special case, the axiom of Rota-Baxter operators on VOAs coincides with X.~Xu's definition of $R$-matrices for VOAs~\mcite{X}, as a vertex operator algebra analog of the $r$-matrices
for Lie algebras as solutions to the CYBE. In a separate work
~\mcite{BGL}, we introduced the tensor form of the CYBE for VOAs.

A closely related notion to the (relative) Rota-Baxter associative
algebra is the dendriform algebra, introduced by Loday from the
periodicity of $K$-theory and plays  an essential
role of splitting of the associativity. In fact, the (tri)dendriform algebra can be
naturally derived from Rota-Baxter algebra. It can further be
characterized by (relative) Rota-Baxter algebras
in terms of representations (see~\cite{BGN3} for example). Since a
vertex algebra can be viewed as a generalization of both the commutative differential algebra and the Lie algebra, it is natural to investigate the analogs of these
structures and their relations with (relative) Rota-Baxter
operators in the context of vertex (operator) algebras.

\delete{\cm{In fact, the main content of this paper is not related to relative RBOs, and hence I suggest that maybe we will not mention such a notion in Introduction and the related parts are revised accordingly.}

\lir{It is okay to mention it in the introduction. But how is this analog studied? There were some discussions on this as a key justification of the nation of vertex dendriform algebras. It should be emphasized or clarified. }
\jq{I'm not sure, professors, but we did use the notion of relative RBO for VOAs, see Corollary~\mref{co:denrb}. }}

\vspace{-.3cm}
\subsection{Layout of the paper}
In Section~\mref{sec:diff}, we first recall the basic notions of
vertex (operator) algebras with emphasis on their derivational
characterizations. We then introduce a notion of $\la$-derivations
on VOAs, as a preparation for the next section. We will prove that the set of $1$-differentials on a
simple VOA $V$ can be identified with the automorphism group of
$V$.

In Section~\mref{sec:rbo}, we define Rota-Baxter operators (RBOs) on VOAs, provide examples and study their relation with
differential operators. To allow more flexibility for examples and
applications, we also introduce variations of RBOs such as weak local 
RBOs, weak global RBOs, and ordinary local RBOs. Homogeneous RBOs on certain CFT-type VOAs
are classified. Xu's theorem on $R$-matrix and modified Yang-Baxter equations for VOAs \mcite{X} was reinterpreted as a special case of a general theorem in Rota-Baxter vertex algebras (Theorem~\mref{thm2.18}).

%\cm{The last sentence might be wrong. Theorem~\mref{thm2.18} does
%not involve CYBE theory at all. So I think that it should be
%deleted.}

%\lir{Leave these to Jianqi to make precise. A better summary is needed. }
%\jq{I also made a remark by the end of Theorem~\mref{thm2.18}}

%The Atkinson additive factorization of RBAs is also extended to VOAs.

Section~\mref{sec:dend} is devoted to extending the notion of dendriform
algebras to VOAs. We define dendriform field and vertex (Leibniz)
algebras and show that they fulfill the splitting property of the
operations in VOAs (Theorem~\mref{thm4.5}). We also obtained some
new kinds of Jacobi identities from our definition of the
dendriform vertex algebras (Theorem~\mref{thm4.10}). Furthermore,
the characterization of dendriform algebras in terms of module structures and relative
Rota-Baxter operators is transported to VOAs
(see Proposition~\mref{prop4.15},
Corollary~\mref{co:denrb} and Proposition~\mref{prop4.19}),
further justify our notion of dendriform vertex algebras.

%\cm{Jianqi, please check the labelling of the cited results}\jq{the labels are correct.}

\smallskip
\noindent
{\bf Conventions.} Throughout this paper, all vector spaces are defined over $\C$. $\N$ denotes the set of nonnegative integers, $\Q$ denotes the set of rational numbers. 

\vspace{-.2cm}
\section{Differential operators on VOAs}
\mlabel{sec:diff}
In this section, we first recall the notions of vertex (operator) algebras and some related definitions. Then we introduce the notion of $\la$-differential operators on VOAs and discuss some of its properties. The $\la$-differential operators are closely related to the Rota-Baxter operators to be studied in the next section.
\vspace{-.2cm}
\subsection{The definition of vertex (operator) algebras}
We briefly recall the background on VOAs that will be needed in the sequel and refer the readers to~\mcite{Bor,FHL,FLM,Hu97,Ka,LL} for further details.

\begin{df}\mlabel{df:va}
A \name{vertex algebra} is a triple $(V, Y, {\bf 1} )$ consisting of a vector space $V$,
a linear map $Y$ called the \name{vertex operator} or the state-field correspondence:
        \begin{align*}
Y:& V \to (\mbox{End}\,V)[[z,z^{-1}]] , \quad
 a\mapsto Y(a,z)=\sum_{n\in{\Z}}a_nz^{-n-1}\ \ \ \  (a_n  \in
\mbox{End}\,V),\nonumber
    \end{align*}
and a distinguished element ${\bf 1} \in V$ called the \name{vacuum vector}, satisfying the following conditions:
        \begin{enumerate}
\item  (Truncation property) For any given $a,b\in V$, $ a_nb=0$ when $n$ is sufficiently large. %$n\gg 0$.
\mlabel{it:trunc}
\item  (Vacuum property) $Y(\vac,z)=\Id_{V}$.
\mlabel{it:vac}
\item  (Creation property) For $a\in V$, we have $Y(a,z){\bf 1}\in V[[z]]$ and $ \lim\limits_{z\to 0}Y(a,z){\bf 1}=a$.
\mlabel{it:crea}
\item (The Jacobi identity) For $a,b\in V$,
        \begin{align*}& \displaystyle{z^{-1}_0\delta\left(\frac{z_1-z_2}{z_0}\right)
        Y(a,z_1)Y(b,z_2)-z^{-1}_0\delta\left(\frac{-z_2+z_1}{z_0}\right)
        Y(b,z_2)Y(a,z_1)}
    \displaystyle{=z_2^{-1}\delta
        \left(\frac{z_1-z_0}{z_2}\right)
        Y(Y(a,z_0)b,z_2)}.
\end{align*}
Here $\delta(x):=\sum_{n\in \Z} x^n$ is the formal delta function.
\mlabel{it:jac}
\end{enumerate}
    \end{df}

The following equivalent characterization of the Jacobi identity is useful in our later discussions (cf. \mcite{LL}; see also \mcite{FHL}).
    \begin{thm}\mlabel{thm2.2}
A vertex algebra $(V,Y,\vac)$ satisfies the following properties.
\begin{enumerate}
\item $($weak commutativity$)$ For $a,b\in V$, there exists $k\in \N$ such that
\begin{equation}\mlabel{2.1}
                (z_{1}-z_{2})^{k} Y(a,z_{1})Y(b,z_{2})=(z_{1}-z_{2})^{k} Y(b,z_{2})Y(a,z_{1}).
            \end{equation}
\item $($weak associativity$)$ For $a,b,c\in V$, there exists $k\in \N$ $($depending on $a$ and $c$$)$ such that
            \begin{equation}\mlabel{2.2}
                (z_{0}+z_{2})^{k} Y(Y(a,z_{0})b,z_{2})c=(z_{0}+z_{2})^{k} Y(a,z_{0}+z_{2})Y(b,z_{2})c.
            \end{equation}
        \end{enumerate}
        Moreover, if $Y:V\ra (\End V)[[z,z^{-1}]] $ is a linear map that satisfies the truncation property, then the Jacobi identity of $Y$ in the definition of a vertex algebra is equivalent to the weak commutativity together with the weak associativity.
    \end{thm}

    Let $(V,Y,\vac)$ be a vertex algebra. Define a translation operator $D :V\ra V$ by letting
    $D a:=a_{-2}\vac,$
    for all $a\in V$. Then $(V,Y,D,\vac)$ satisfies the \name{$D $-derivative property}:
    \begin{equation}\mlabel{2.3}
        Y(D a,z)=\frac{d}{dz}Y(a,z),
    \end{equation}
    the \name{$D $-bracket derivative property}:
    \vspace{-.2cm}
    \begin{equation}\mlabel{2.4}
        [D , Y(a,z)]=\frac{d}{dz}Y(a,z),
    \end{equation}
    and the \name{skew-symmetry} formula:
    \vspace{-.2cm}
    \begin{equation}\mlabel{2.5}
        Y(a,z)b=e^{z D} Y(b,-z)a, \quad  a, b\in V.
    \end{equation}
Equalities \meqref{2.3} and \meqref{2.4} together are called the \name{$D$-translation invariance property}.
On the other hand, a vertex algebra also has the following equivalent condition (cf. \mcite{LL}).
\begin{thm}\mlabel{thm2.3}
Consider a triple $(V,Y,\vac)$ consisting of a vector space $V$, a
linear map $Y:V\to (\End V)[[z,z^{-1}]]$ and a distinguished
vector $\vac$. The triple is a vertex algebra if and
only if it satisfies the truncation, vacuum and creation
properties \meqref{it:trunc}--\meqref{it:crea}, the weak
commutativity \meqref{2.1}, and allows a derivation $D$ that
satisfies the $D$-bracket derivative property \meqref{2.4}.
\end{thm}
Our later discussion needs the following weaker notions of vertex algebras introduced in \mcite{LTW}.

\begin{df}\mlabel{df2.10}
A \name{vertex Leibniz algebra} $(V,Y)$ is a vector space $V$ equipped with a linear map $Y: V\ra (\End V)[[z,z^{-1}]]$, satisfying the truncation property and the Jacobi identity.
\end{df}

As a special case, a subspace $U$ of a vertex algebra $(V,Y,\vac)$ is a vertex Leibniz subalgebra with respect to the restricted vertex operator $Y|_{U}$ if it satisfies $a_{n}b\in U$, for all $a,b\in U$, and $n\in \Z$. A related notion is the vertex algebra without vacuum (see \mcite{HL}):

\begin{df}\mlabel{df2.5}
A \name{vertex algebra without vacuum} is a vector space $V$, equipped with a linear map $Y:V\ra (\End V)[[z,z^{-1}]]$ and a linear operator $D:V\ra V$ satisfying the truncation property, the Jacobi identity, the $D$-derivative property \meqref{2.3}, and the skew-symmetry \meqref{2.5}. We denote a vertex algebra without vacuum by $(V,Y,D)$.
    \end{df}

The following fact is proved by Huang and Lepowsky in \mcite{HL}; see also \mcite{LTW}:
\begin{prop}\mlabel{prop2.6}
Let $V$ be a vector space, equipped with a linear map  $Y: V\ra (\End V)[[z,z^{-1}]]$, satisfying the truncation property. If $D:V\ra V$ is another linear map that satisfies the $D$-bracket derivative property \meqref{2.4} and skew-symmetry \meqref{2.5}, then the weak commutativity \meqref{2.1} of $Y$ follows from the weak associativity \meqref{2.2}.
\end{prop}

\begin{df}\mlabel{df2.2}
A \name{vertex operator algebra (VOA)} is a quadruple $(V, Y, \vac, \omega)$, where
\begin{itemize}
\item $(V,Y(\cdot, z), \vac)$ is a $\Z$-graded vertex algebra: $V=\bigoplus_{n \in \Z}V_n,$
        such that ${\bf 1} \in V_0$, $\text{dim}V_n < \infty$ for each $n\in \Z$, and $V_n=0$ for $n$ sufficiently small,
\item ${\omega} \in V_2$ is another distinguished element, called the \name{Virasoro element} and for which we denote $Y(\omega,z)=\sum_{n\in\Z}L(n)z^{-n-2}$, so $L(n)=\omega_{n+1}$, $n\in \Z$. It satisfies the following additional conditions.
\end{itemize}
\begin{enumerate}
\item[(5)] (The Virasoro relation)
$[L(m),L(n)]=(m-n)L(m+n)+\frac{1}{12}(m^3-m)\delta_{m+n,0}c,$
where $c\in \C$ is called the central charge (or rank) of $V$.
\item[(6)] ($L(-1)$-derivation property) $D=L(-1)$, and  $$\frac{d}{dz}Y(v,z)=Y(L(-1)v,z)=[L(-1),Y(v,z)].$$
\item[(7)] ($L(0)$-eigenspace property) $L(0)v=nv$, for all $v\in V_{n}$ and $n\in \Z$.
\end{enumerate}
A VOA $V$ is said to be of the \name{CFT-type}, if $V=V_0\op V_+$, where $V_0=\C\vac$ and $V_+=\bigoplus_{n=1}^\infty V_n$.
\end{df}

We also need the notion of modules over VOAs (cf. \mcite{FHL,FLM,LL}).
\begin{df}\mlabel{df2.13}
Let $(V,Y,\vac,\om)$ be a VOA, a \name{weak $V$-module} $(W,Y_{W})$ is a vector space $W$ equipped with a linear map
\vspace{-.2cm}
$$
Y_{W}:V\ra (\End W)[[z,z^{-1}]],\quad
    a\mapsto Y_{W}(a,z)=\sum_{n\in \Z} a_{n} z^{-n-1},
\vspace{-.2cm}
$$
satisfying the following axioms.
\begin{enumerate}
    \item (Truncation property) For $a\in V$ and $v\in W$, $a_{n}v=0$ for $n$ sufficiently large. %$n\gg 0$.
            \item (Vacuum property) $Y_{W}(\vac,z)=\Id_{W}$.
            \item (The Jacobi identity) For $a,b\in V$, and $u\in W$
            \begin{align*}
                &z_{0}^{-1}\delta\left(\frac{z_{1}-z_{2}}{z_{0}}\right) Y_{W}(a,z_{1})Y_{W}(b,z_{2})u-z_{0}^{-1}\delta\left(\frac{-z_{2}+z_{1}}{z_{0}}\right)Y_{W}(b,z_{2})Y_{W}(a,z_{1})u\\
                &=z_{2}^{-1}\delta \left(\frac{z_{1}-z_{0}}{z_{2}}\right) Y_{W}(Y(a,z_{0})b,z_{2})u.
            \end{align*}
        \end{enumerate}

A weak $V$-module $W$ is called \name{admissible (or $\N$-gradable)} if $W=\bigoplus_{n\in \N}W(n)$, with $\dim W(n)<\infty$ for each $n\in \N$, and $a_{m}W(n)\subset W(\wt(a)-m-1+n)$ for homogeneous $a\in V$, $m\in \Z$, and $n\in \N$.

An \name{ordinary $V$-module} is an admissible $V$-module $W$ such that each $W(n)$ is an eigenspace of the operator $L(0)=\Res z Y_W(\om ,z)$ of eigenvalue $\la+n$, where $\la\in \Q$ is a fixed number called the conformal weight of $W$. 
    \end{df}
Let $(W,Y_{W})$ be a weak module over a VOA $V$. Write $Y_{W}(\om ,z)=\sum_{n\in \Z} L(n)z^{-n-2}$. It is proved in \mcite{DLM1} that $Y_{W}$ also satisfies the $L(-1)$-derivative property (in Definition~\mref{df2.2}) and the $L(-1)$-bracket derivative property in~\meqref{2.4}:
$Y_{W}(L(-1)a,z)=\frac{d}{dz}Y_{W}(a,z)=[L(-1),Y_{W}(a,z)].$

A VOA $(V,Y,\vac,\om)$ is obviously an ordinary module over $V$ itself, with $Y_W=Y$. $(V,Y)$ is called the adjoint $V$-module (see Section 2 in \mcite{FHL}). $V$ is called \name{simple} if the adjoint module $V$ has no proper submodules. 
\vspace{-.3cm}
\subsection{The $\la$-derivations on VOAs}
Recall that a derivation on a VOA $(V,Y,\vac,\om)$ is a linear map
$d:V\ra V$ satisfying $d\vac=0$, $d\om =0$, and
$d(a_nb)=(da)_nb+a_n(db)$ for all $a,b\in V$ and $n\in \Z$ (cf.
\mcite{DG1,H1}). We introduce the following generalized notion of
derivations with weights for VOAs.
\begin{df}\mlabel{df2.12}
Let $(V,Y,\vac)$ be a vertex algebra, and $\la\in \C$ be a fixed complex number. A linear map $d:V\ra V$ is called a \name{weak $\la$-derivation} of $V$ if it satisfies
    \begin{equation}\mlabel{2.11'}
d(Y(a,z)b)=Y(da,z)b+Y(a,z)db+\la Y(da,z)db,\quad a, b\in V.
\end{equation}
In other words,
$  d(a_mb)=(da)_mb+a_m(db)+\la (da)_{m}(db),$ for all $a,b\in V, m\in \Z.$

Let $(V,Y,\vac,\om)$ be a VOA. A \name{ $\la$-derivation} on $V$ is a weak $\la$-derivation $d:V\ra V$ such that $d\vac=0$ and $d\om=0$. The space of $\la$-derivations on $V$ is denoted by $\Diff_{\la}(V)$.
    \end{df}

    In particular, for a $\la$-derivation $d:V\ra V$ on a VOA $(V,Y,\vac,\om)$, since $d\om =0$, by \meqref{2.11'} we have
$d(Y(\om,z)b)=Y(d\om,a)+Y(\om,z)db+\la Y(d\om,z)db$, which implies
$d(Y(\om,z)b)=Y(\om,z) db$ and hence $dL(-1)b=L(-1)db$ for  $b\in V.$
Thus, the $\la$-derivation $d$ is automatically compatible with the derivation $D=L(-1)$ on the VOA: $d L(-1)=L(-1)d$.

By Definition \mref{df2.12}, it is easy to see that a $0$-differential operator is just a derivation on $V$. i.e., $\Diff _{0}(V)=\Der(V)$. The $1$-derivations have a nice correspondence with the automorphisms on $V$. Recall that an endomorphism of $V$ is a linear map $\phi: V\ra V$ such that
\begin{equation}\mlabel{2.6'}
    \phi(Y(a,z)b)=Y(\phi(a),z)\phi(b),
\end{equation}
$\phi(\vac)=\vac$, and $\phi(\om)=\om$ (cf. \cite{FHL}). The space of endomorphisms on $V$ is denoted by $\End(V)$.

\begin{lm}\mlabel{lm2.9}
        If $(V,Y,\vac,\om)$ is a simple VOA, 
        %\cm{simple VOA is not defined. Maybe we give an exact notion at the end of Def. 2.7 or
        %just before this lemma as ``Recall that a VOA $V$ is called {\bf simple} if ..."} 
        then $\End(V)$ is a division algebra over $\C$, with the unit group $\Aut(V)$.
    \end{lm}
\begin{proof}
Let $\phi\in \End_{V}(V)$ and $\phi\neq 0$. It suffices to show that $\phi$ is an automorphism. First we note that $\ker \phi$ is an ideal of $V$: for $u\in \ker\phi$, $a\in V$, and $m\in \Z$, we have $\phi(a_m u)=\phi(a)_m \phi(u)=0$ by \meqref{2.6'}. Since $\phi\neq 0$, we have $\ker \phi=0$, and $\phi$ is injective.
Moreover, for $a\in V_n$, since $\phi(\om)=\om$, we have $L(0)\phi(a)=\phi(\om)_1 \phi(a)=\phi(L(0)a)=n\phi(a)$. It follows that $\phi(V_n)\ssq V_n$ for $n\in \N$. Since $\dim V_n<\infty$, we have $\phi|_{V_n}:V_n\ra V_n$ is a linear isomorphism. Thus, $\phi$ is an automorphism.
\end{proof}

    \begin{prop}\label{prop2.13}
        Let $(V,Y,\vac,\om)$ be a simple VOA. Then the map $\al : \Diff_{1}(V)\ra \Aut(V): d\mapsto d+\Id_{V}$ is a bijection.
    \end{prop}
\begin{proof}
Since $d\vac=0$ and $d\om=0$, we have $(d+\Id_{V})(\vac)=\vac$ and $(d+\Id_{V})(\om)=\om$. Moreover,
\begin{align*}
(d+\Id)(Y(a,z)b)&=Y(da,z)b+Y(a,z)db+Y(da,z)db+Y(a,z)b\\
&=Y(da+a,z)b+Y(a+da,z)db\\
&=Y((d+\Id)(a),z)(d+\Id)(b).
\end{align*}
Thus $d+\Id_{V}\in \End_{V}(V)$. But $d+\Id_{V}\neq 0$, since
otherwise $d=-\Id_{V}$ does not satisfy $d\vac =0$. Hence $\al
(d)=d+\Id_{V}$ is in  $(\End_{V}(V))^{\times}=\Aut(V)$, in view of
Lemma \mref{lm2.9}.

On the other hand, for $\phi\in \Aut(V)$, we have $(\phi-\Id)(\vac)=0$ and $(\phi-\Id)(\om)=0$, and
\begin{align*}
&Y((\phi-\Id)(a),z)b+Y(a,z)(\phi-\Id)(b)+Y((\phi-\Id)(a),z)(\phi-\Id)(b)\\
&=Y(\phi(a),z)b-Y(a,z)b+Y(a,z)\phi(b)-Y(a,z)b+Y(\phi(a),z)\phi(b)-Y(a,z)\phi(b)\\
&\ \ -Y(\phi(a),z)b+Y(a,z)b\\
%&=Y(\phi(a),z)\phi(b)-Y(a,z)b\\
&=(\phi-\Id)(Y(a,z)b).
\end{align*}
Thus, $\phi-\Id$ is in $\Diff_{1}(V)$. Clearly, $\phi \mapsto \phi -\Id_{V}$ is the inverse of $\al $. Hence $\al $ is a bijection.
    \end{proof}

\section{Rota-Baxter operators on vertex algebras}
\mlabel{sec:rbo}
In this section, we introduce various types of Rota-Baxter operators on vertex algebras and the notion of Rota-Baxter vertex (operator) algebras. We will give a few examples of Rota-Baxter vertex algebras, some of which generalize the classical Rota-Baxter associative algebras and the Rota-Baxter Lie algebras.

\subsection{Definition and first properties of Rota-Baxter operators on vertex algebras}
Recall that a \name{Rota-Baxter (associative) algebra} of weight $\la\in \C$ (cf. \mcite{G}) is an associative algebra $(R,\cdot)$, equipped with a linear map $P:R\ra R$, called the \name{Rota-Baxter operator (RBO)}, satisfying
$$P(a)\cdot P(b)=P(P(a)\cdot b)+P(a\cdot P(b))+\la P(a\cdot b),\quad a, b\in R.$$
Similarly, a \name{Rota-Baxter Lie algebra} of weight $\la$ is a Lie algebra $(\g,[\cdot,\cdot])$ equipped with a linear map $P:\g\ra \g$ such that
$$[P(a),P(b)]=P([P(a),b]) +P([a,P(b)])+\la P([a,b]),
\quad a, b\in \g.$$
The Rota-Baxter operator can be defined over various other algebraic structures~\mcite{BBGN,BGP}.

The vertex operator $Y$ can be viewed as a generating Laurent series of the infinitely many component binary operations $a_mb, m\in Z$. We introduce the notion of Rota-Baxter operators on vertex algebras first for the component operations and then for the vertex operators.

\begin{df}\mlabel{df2.7}
Let $(V,Y,\bf{1})$ be a vertex algebra, $\lambda\in \C$ be a fixed complex number, and $m\in \Z$.
\begin{enumerate}
\item An \name{$m$-(ordinary) Rota-Baxter operator ($m$-(ordinary) RBO) on $V$ of weight
$\lambda$} is a linear map $P:V\rightarrow V$, satisfying the following condition:
\begin{equation}\mlabel{2.6}
(Pa)_{m}(Pb)=P(a_{m}(Pb))+P((Pa)_{m}b)+\la P(a_{m}b), \quad a, b\in V.
\end{equation}
We denote the set of $m$-RBOs by $\RBO(V)(m)$.
\item An \name{(ordinary) RBO on $V$ of weight $\la$} is a linear map $P:V\ra V$ satisfying \meqref{2.6} for every $m\in \Z$. In other words, $P$ satisfies the following condition:
\begin{equation}\mlabel{2.7}
Y(Pa,z)Pb=P(Y(Pa,z)b)+P(Y(a,z)Pb)+\la P(Y(a,z)b), \quad a, b\in V.
\end{equation}
We denote the set of RBOs on $V$ by $\RBO(V)=\bigcap_{m\in \Z} \RBO(V)(m)$.
\item An ($m$-)ordinary RBO $P$ on $V$ is called \name{translation invariant} if $PD=DP$, where $D$ is the translation operator: $Da=a_{-2}\vac$.
\item Let $V$ be a VOA, and let $P$ be an ($m$-)ordinary RBO. $P$ is called \name{homogeneous of degree $N$} if $P(V_n)\ssq V_{n+N}$ for all $n\in \N$. Degree zero RBOs are called \name{level preserving}.
\end{enumerate}
A \name{Rota-Baxter vertex algebra (RBVA)} is a vertex algebra
$(V,Y,\vac)$, equipped with an ordinary RBO $P:V\ra V$ of weight
$\la$. We denote such an algebra by $(V,Y,\vac,P)$. We can
similarly define a Rota-Baxter vertex operator algebra
(RBVOA) $(V,Y,\vac,\om,P)$.
\end{df}

\begin{remark}
The notions of $m$-RBOs and RBOs on VOAs
are closely related to the tensor form of the
classical Yang-Baxter equations for VOAs; see \mcite{BGL} for
more details.
\end{remark}

It is clear that for any vertex algebra $V$, $P=-\la \Id_{V}$ satisfies \meqref{2.7}. Hence any vertex algebra can be viewed as an RBVA trivially in this way.

Let $(V,Y,\vac,\om,P)$ be an RBVOA of weight $\la$. Recall that (cf. \mcite{Bor}) the first level $\g=V_{1}$ is a Lie algebra, with the Lie bracket $[a,b]=a_{0}b,$
for all $a,b\in \g$. Then it follows from \meqref{2.7} that $(\g, P|_{\g})$ is a Rota-Baxter Lie algebra. Conversely, if $p:\g\ra \g$ is an RBO of the Lie algebra $\g$, and $\g$ is the first level $V_1$ of a VOA $V$, then $p$ can be easily extended to a $0$-ordinary RBO $P:V\ra V$ by letting $P|_{V_1}=p$ and $P(V_n)=0$, for all $n\neq 1$; see Example 2.10 in \mcite{BGL}.

Our definition of the Rota-Baxter operators for vertex algebras
is similar to the $R$-matrix for VOAs in \mcite{X} where an {\bf $R$-matrix}  for a VOA $(V,Y,\vac,\om)$ is defined to be a linear map $R:V\ra V$ such that $[R,L(-1)]=0,$
and $Y_{R}:V\ra (\End V)[[z,z^{-1}]]$ defined by
\begin{equation} \mlabel{eq:myb}    Y_{R}(a,z)b:=Y(Ra,z)b+Y(a,z)Rb
\end{equation}
satisfies the Jacobi identity. The following conclusion is proved by Xu in \mcite{X}.

\begin{prop} \mlabel{prop3.3} Let $(V,Y,\vac,\om)$ be a
VOA. If a linear map $R:V\ra V$ satisfies $[R,L(-1)]=0$ and the
so called ``modified Yang-Baxter equation''
\begin{equation}\mlabel{2.9}
    Y(Ra,z)Rb-R(Y_{R}(a,z)b)=\la Y(a,z)b,
\end{equation}
where $\la=0$ or $-1$, then $R$ is an $R$-matrix for $V$.
\end{prop}

Note that, in view of \meqref{eq:myb}, $R$ satisfying \meqref{2.9} with $\la=0$ is the special case of a Rota-Baxter operator of weight $0$ on $V$ defined in \meqref{2.7}. Hence a linear map $P:V\ra V$ on a VOA $V$ is a translation invariant RBO $P$ of weight $0$ if and only if it is an $R$-matrix of $V$ in the sense of \mcite{X}.

On the other hand, the identity \meqref{2.9} with $\la =-1$ for $P$ is equivalent to the Rota-Baxter identity \meqref{2.7} with $\la = -1$ for the operator $Q=(\Id -2P)/2$ by a direct calculation, as in the classical cases of associative algebras and Lie algebras~(see~\mcite{G}).

\smallskip

Let $(A,d,1_{A})$ be a commutative unital differential algebra. Recall that $A$ is a vertex algebra (cf. \mcite{Bor}) with the vertex operator $Y$ given by
\begin{equation}\mlabel{3.11}
    Y(a,z)b:=(e^{zd}a)\cdot b,\quad a,b\in A,
\end{equation}
and $\vac:=1_{A}$. The differential operator $d$ of $A$ is the translation operator $D $ in \meqref{2.4}. In particular, let $V=\C[t]$ be the polynomial algebra with variable $t$. Define
\begin{equation}\mlabel{2.10}
Y(t^{m},z)t^{n}:=(e^{z\frac{d}{dt}}t^{m})\cdot t^{n}=\sum_{j\geq 0} \binom{m}{j} t^{m+n-j}z^{j}, \quad m,n \in \N.
\end{equation}
Then $(\C[t],Y,1)$ is a vertex algebra.

\begin{prop} \mlabel{prop3.4}
Let $P:\C[t]\ra \C[t]$ be the usual (integration) Rota-Baxter operator on $\C[t]$:
$$P(t^{m})=\int_{0}^{t}s^{m}ds=\frac{t^{m+1}}{m+1}\quad m\in \N.$$
Then $(\C[t],Y,1,P)$ is an RBVA of weight $0$
\end{prop}

We give this result as a special case of a more general setting. The divided power algebra is the vector space  $A=\bigoplus_{m=0}^\infty \C t_m,$
equipped with the product
$t_m\cdot t_n:=\binom{m+n}{n}t_{m+n}, \quad m,n\in \N.$ Then $\cdot$ is commutative associative with unit $1_A=t_0$.
Note that $d:A\ra A,\ d(t_m)=t_{m-1}$ is a derivation on $A$, and so $(A,Y,1_A)$ is a vertex algebra, with
\begin{equation}\mlabel{2.12''}
Y(t_m,z)t_n:=(e^{zd}t_m)\cdot t_n=\sum_{j\geq 0}\frac{(m+n-j)!}{(m-j)!n!j!}t_{m+n-j}z^j, \quad m, n\in N.
\end{equation}

\begin{prop} \mlabel{prop3.5}
Define
$$P:A\ra A, \quad P(t_m):=t_{m+1}, m\in \N.$$
Then $(A, Y, 1_A, P)$ is an RBVA of weight $0$.
\end{prop}

As is well known, the polynomial algebra $\C[t]$, with the basis $t_n:=t^n/n!$, is a (realization of the) divided power algebra.  Thus Proposition~\mref{prop3.4} is a consequence of Proposition~\mref{prop3.5}.
\begin{proof}[Proof of Proposition~\mref{prop3.5}]
For $m,n\in \N$, by \meqref{2.12''} we have
{\small \begin{align*}
Y(Pt_{m},z)Pt_{n}&=Y(t_{m+1},z)t_{n+1}=\sum_{j\geq 0}\frac{(m+n+2-j)!}{(m+1-j)!(n+1)!j!}t_{m+n+2-j}z^j,\\
P(Y(Pt_{m},z)t_{n})&=P(Y(t_{m+1},z)t_n)=P\bigg(\sum_{j\geq 0}\frac{(m+n+1-j)!}{(m+1-j)!n!j!}t_{m+n+1-j}z^j\bigg)\\
&=\sum_{j\geq 0}\frac{(m+n+1-j)!}{(m+1-j)!n!j!}t_{m+n+2-j}z^j,\\
P(Y(t_{m},z)Pt_{n})&=P(Y(t_m,z)t_{n+1})=P\bigg(\sum_{j\geq 0}\frac{(m+n+1-j)!}{(m-j)!(n+1)!j!}t_{m+n+1-j}z^j\bigg)\\
&=\sum_{j\geq 0}\frac{(m+n+1-j)!}{(m-j)!(n+1)!j!}t_{m+n+2-j}z^j.
\end{align*}}
Since we have
{\small \begin{align*}
&\sum_{j\geq 0}\frac{(m+n+1-j)!}{(m+1-j)!n!j!}+\sum_{j\geq 0}\frac{(m+n+1-j)!}{(m-j)!(n+1)!j!}\\
&=\sum_{j\geq 0}\frac{(m+n+1-j)!(n+1)+(m+n+1-j)!(m+1-j)}{(m+1-j)!(n+1)!j!}\\
&=\sum_{j\geq 0}\frac{(m+n+1-j)!(m+n+2-j)}{(m+1-j)!(n+1)!j!}=\sum_{j\geq 0}\frac{(m+n+2-j)!}{(m+1-j)!(n+1)!j!},
\end{align*}}
it follows that $Y(Pt_m,z)Pt_n=P(Y(Pt_m,z)t_n)+P(Y(t_m,z)Pt_n)$. Hence $(A,Y,1_A,P)$ is an RBVA of weight $0$.
\end{proof}

Both $(\C[t],\frac{d}{dt},P,1_{\C[t]})$ and $(A=\bigoplus_{m=0}^\infty \C t_m,d,P,1_A)$ are special cases of the commutative unital differential Rota-Baxter algebras $(A,d,P,1_A)$. The latter means that, $d$ is a derivation on $A$, $P$ is an RBO of weight $0$ on $A$, and $d\circ P=\Id_A$; see \mcite{GK} for more details. We have the following property. See Corollary~\mref{coro3.10} for another related result.

\begin{prop}
Let $(A,d,P,1_A)$ be an commutative unital differential Rota-Baxter algebra, and let $Y(a,z)b=(e^{zd}a)\cdot b$. Then we have:
$$Y(Pa,z)Pb-P(Y(Pa,z)b)-P(Y(a,z)Pb)\in (\ker d)[[z]], \quad a, b\in V.$$
In particular, $(A,Y,1_{A},P)$ is an RBVA of weight $0$ if $\ker d=0$.
\end{prop}
\begin{proof}
First we note that  $P(a)_{-1}P(b)=P(P(a)_{-1}b)+P(a_{-1}P(b))$
for all $a,b\in A$, since the product of  $A$ is given by
$x\cdot y=x_{-1}y$ for all $x,y\in A$.

Now assume that $n\geq 1$. By \meqref{2.4} we have $d(a_{-n}b)=(da)_{-n}b+a_{-n}db$ and  $(da)_{-n}=nda_{-n-1}$. Moreover, $a-Pd(a)\in \ker d$ for all $a,b\in A$ as $d\circ P=\Id_{A}$, hence we have:
\begin{align*}
&nP(a)_{-n-1}P(b)-nP(P(a)_{-n-1}b)-nP(a_{-n-1}P(b))\\
&=(dP(a))_{-n}P(b)-P((dP(a))_{-n}b)-P((da)_{-n}P(b))\\
&=a_{-n}P(b)-Pd(a_{-n}P(b))\\
&\equiv 0\pmod{\ker d}.
\end{align*}
This finishes the proof because $Y(a,z)b=\sum_{n\geq 0} (a_{-n-1}b)z^{n}$.
\end{proof}

We also have the notion of relative Rota-Baxter operators
 introduced in \mcite{BGL}, as the operator form of
the classical Yang-Baxter equation for VOAs.
\begin{df}
Let $(V,Y,\vac,\om)$ be a VOA and $(W,Y_W)$ be a weak $V$-module. A \name{relative Rota-Baxter operator (relative RBO)} is a linear map $T:W\ra V$ such that
  \begin{equation}\mlabel{4.60'}
    Y(Tu,z)Tv=T(Y_{W}(Tu,z)v)+T(Y_{WV}^W(u,z)Tv), \quad u,v\in W.
    \end{equation}
\end{df}
In Section~\mref{sec:dend}, we will use the dendriform vertex algebra structure to give examples of relative Rota-Baxter operators on vertex algebras without vacuum.

\subsection{The $\la$-derivations and the weak local Rota-Baxter operators}

Propositions \mref{prop3.4} and \mref{prop3.5} indicate that a
right inverse $P$ of the translation operator $D$ on certain
commutative vertex algebras can give rise to an ordinary RBO of weight $0$.

However, in the case of non-commutative VOAs, the translation operator $D=L(-1)$ and most of the derivations are {\em not} invertible globally---they only admit local inverses. On the other hand, by its defining identity \meqref{2.7}, if $P:V\ra V$ is an RBO, then we must have $P(a)_m P(b)\in P(V)$ for all $a,b\in V$ and $m\in \Z$. i.e., $P(V)\ssq V$ is a vertex Leibniz subalgebra (see Definition \mref{df2.10}). This is also a strong condition imposed on $P$. If we weaken these conditions, we can expect to construct more examples of Rota-Baxter type operators from the ``right inverses" of $\la$-derivations on vertex algebra $V$ on a suitable domain.

\begin{df}\mlabel{df3.5}
Let $(V,Y,\bf{1})$ be a vertex algebra, $\lambda\in \C$ be a fixed complex number, and $U\subset V$ be a linear subspace.
\begin{enumerate}
\item
A \name{ weak local Rota-Baxter operator (RBO) on $U$ of weight $\lambda$} is a linear map $P:U\rightarrow V$, satisfying the following condition. Whenever $a,b\in U$ and $m\in \Z$ such that $P(a)_mP(b)\in P(U)$, one has $a_{m}(Pb)+(Pa)_{m}b+\la a_{m}b\in U$, and
\begin{equation}\mlabel{2.6'b}
(Pa)_{m}(Pb)=P\Big(a_{m}(Pb)+(Pa)_{m}b+\la a_{m}b\Big).
\end{equation}
A \name{ weak global RBO of weight $\la$} is a weak local RBO on $V$. 
\item An \name{ ordinary local RBO on $U$ of weight $\la$} is a weak local RBO $P:U\ra V$ of weight $\la$ such that $P(U)$ is a vertex Leibniz subalgebra of $V$. In other words, $P:U\ra V$ is a linear map satisfying:
\begin{equation}\mlabel{2.7'}
Y(Pa,z)Pb=P\Big(Y(Pa,z)b+Y(a,z)Pb+\la Y(a,z)b\Big), \quad a, b\in U.
\end{equation}
An \name{ordinary global RBO of weight $\la$} is an ordinary local RBO on $V$. 
In particular, by equation \meqref{2.7'}, the notion of ordinary global RBOs is the same as the notion of ordinary RBOs in Definition \mref{df2.7}.
\end{enumerate}

A local RBO (weak or ordinary) $P:U\ra V$ is called \name{ translation invariant}, if $DU\ssq U$ and $PD=DP$ on $U$. Let $V$ be a VOA, and let $P:U\ra V$ be a local RBO. Then $P$ is called \name{ homogeneous of degree $N$}, if $U\subset V$ is a homogeneous subspace: $U=\bigoplus_{n=0}^\infty U_n$, and $P(U_n)\ssq V_{n+N}$ for all $n\in \N$.
\end{df}

\begin{remark}
There are key distinctions between a weak local
RBO $P$ and an ordinary RBO. One is that here $P$ is defined
locally on a subspace $U$ of $V$. The other one is that in
equations \meqref{2.6'b} and \meqref{2.7'}, we do not require
$P(a)_m b$ or $a_mP(b)$ to be individually contained in the domain
$U$ of $P$. So their right hand sides cannot be separated into
three terms, as in \meqref{2.6} and \meqref{2.7}. 
The following diagram illustrated the relations between these concepts:
$$
\begin{tikzcd}
\{\text{ordinary\ global\ RBOs}\}\arrow[r,hook,"\text{subset}"] \arrow[d,hook,"\text{subset}"]& \{\text{weak\ global\ RBOs}\}\arrow[d,hook,"\text{subset}"]\\
\{\text{ordinary\ local\ RBOs}\}\arrow[r,hook,"\text{subset}"]& \{\text{weak\ local\ RBOs}\}
\end{tikzcd}
$$
\end{remark}

We extend the following properties on RBOs to weak and ordinary local RBOs for later use.

\begin{prop}\mlabel{prop2.9}
Let $(V,Y,\vac)$ be a vertex algebra, and $U\subset V$ be a linear subspace.
\begin{enumerate}
\item If $P:U\ra V$ is a weak (resp. ordinary) local RBO on $U$ of weight $\la\neq 0$, then $-P/\la$ is a weak (resp. ordinary) local RBO on $U$ of weight $-1$. If $P:U\ra V$ is a weak (resp. ordinary) local RBO of weight $1$, then $\la P$ is a weak (resp. ordinary) local RBO of weight $\la$.
\mlabel{it:locrescale}
\item Let $P$ be an ordinary local RBO on $U$ of weight $\la$. Then $\tilde{P}=-\la \Id_{V}-P$ is an ordinary local RBO on $U$ of weight $\la$.
\mlabel{it:locadj}
\end{enumerate}
\end{prop}
\begin{proof}
\meqref{it:locrescale} Let $P:U\ra V$ be a weak local RBO of weight $\la\neq 0$. Let $a,b\in U$ and $n\in \Z$ satisfy $(-P/\la)(a)_n(-P/\la)(b)\in (-P/\la)(U)=P(U)$. Then $P(a)_nP(b)\in P(U)$, and by Definition \ref{df2.7}, we have $a_nP(b)+(Pa)_nb+\la  a_nb\in U$, and $(Pa)_{n}(Pb)=P(a_{n}(Pb)+(Pa)_{n}b+\la a_{n}b).$ It follows that $a_n(-P/\la)(b)+((-P/\la)(a))_nb-  a_nb\in U$ and
$$((-P/\la)(a))_{n}((-P/\la)(b))=(-P/\la)(a_{n}(-P/\la)(b)+((-P/\la)(a))_{n}b- a_{n}b).$$
Thus, $-P/\la:U\ra V$ is a weak local RBO of weight $-1$. The proof of the rest for \meqref{it:locrescale} and \meqref{it:locadj} is similar.
\delete{
\meqref{it:locadj}
Since $P:U\ra V$ is an ordinary RBO, by Definition \ref{df2.7}, we have $a_{n}(Pb)+(Pa)_{n}b+\la a_{n}b\in U$ for all $a,b\in U$ and $n\in \Z$. It follows that
$$a_n(-\la-P)(b)+(-\la-P)(a)_nb+\la a_nb=-a_nPb-(Pa)_nb-\la a_nb\in U,$$
and it is easy to check that
$$(-\la-P)(a)_n(-\la-P)(b)=(-\la-P)(a_n(-\la-P)(b)+(-\la-P)(a)_nb+\la a_nb), \quad a, b\in U, n \in Z.$$
Thus $(-\la-P)(a)_n(-\la-P)(b)\in (-\la-P)(U)$ for all $a,b\in U$ and $n\in \Z$, and $\tilde{P}=-\la-P$ satisfies \eqref{2.7}. This shows that $\tilde{P}$ is an  ordinary local RBO on $U$ of weight $\la$.
}
\end{proof}

An immediate advantage of the local RBOs is that a local inverse of a weak $\la$-derivation (see Definition \mref{df2.12}) of a vertex algebra gives rise to a weak local RBO of weight $\la$.

\begin{prop}\mlabel{prop2.14}
Let $(V,Y,\vac)$ be a vertex algebra, and let $d:V\ra V$ be a weak $\la$-derivation. Suppose that there exists a linear map $P:U(:=dV)\ra V$ such that $d\circ P=\Id_U$. Then $P:U\ra V$ is a weak local RBO on $U$ of weight $\la$.
\end{prop}
\begin{proof}
Let $a,b\in U$ and $n\in \Z$ satisfy $(Pa)_n(Pb)
=P(c)\in P(U)$. Then we have
\begin{align*}
dP(c)&=d((Pa)_n(Pb))=(dP)(a)_n(Pb)+(Pa)_n(dPb)+\la (dPa)_n(dPb)\\
&=a_n(Pb)+(Pa)_nb+\la a_nb,
\end{align*}
and $dP(c)=c$ since $dP=\Id_U$. Thus, $a_n(Pb)+(Pa)_nb+\la a_nb=c\in U$, and
$$(Pa)_n(Pb)= PdP(c)=P(a_n(Pb)+(Pa)_nb+\la a_nb).$$
Hence $P:U\ra V$ satisfies \eqref{2.6}, and so $P$ is a weak local RBO on $U$ of weight $\la$.
\end{proof}

\begin{coro}\mlabel{coro3.10}
Let $(A,d,P,1_A)$ be an commutative unital differential Rota-Baxter algebra. Then the vertex algebra $(A,Y,1_A,P)$ with $Y$ given by \meqref{3.11} is an RBVA of weight $0$, if $P$ satisfies $P(a)\cdot P(b)\in P(A)$ and $(d^n a)\cdot P(b)\in P(A)$, for all $a,b\in A$ and $n\in \N$.
\end{coro}
\begin{proof}
By \meqref{2.4} and Definition \mref{df2.12}, $d=D:A\ra A$ is an
$0$-derivation of the vertex algebra $(A,Y,1_A)$. Since $d\circ
P=\Id_A$ by the definition of a differential RBA, it follows that  $dA=A$, and $P:A(=dA)\ra A$ is a
weak global RBO of weight $0$ on the vertex algebra $A$ by Proposition
\mref{prop2.14}. If $P$ satisfies the last condition, then
$Y(P(a),z)P(b)=P(a)\cdot P(b)+\sum_{j\geq 1}\frac{1}{j!}
(d^{j-1}a)\cdot P(b)\in P(A)((z)) $, and so $P:A\ra A$ is an
ordinary RBO of weight $0$.
\end{proof}
By \meqref{2.10} and \meqref{2.12''}, it is easy to check that the conditions in Corollary \mref{coro3.10} are satisfied by $(\C[t],\frac{d}{dt}, P,1_{\C[t]})$ and $(A=\bigoplus_{m=0}^\infty \C t_m,d,P,1_A)$. This provides us with another proof of Propositions \mref{prop3.4} and \mref{prop3.5}.

There are many examples of weak $0$-derivations on VOAs. We can use them to construct examples of weak local RBOs on general VOAs by Proposition \mref{prop2.14}.

\begin{example}\label{ex2.15}
Let $(V,Y,\vac,\om)$ be a CFT-type VOA. By the main Theorem in \mcite{DLinM}, the operator $L(-1):V\ra V$ is injective on $V_{+}$. Moreover, we have  $L(-1)\vac=\vac_{-2}\vac=0$, and $L(-1)$ is a weak $0$-derivation by \eqref{2.3} and \eqref{2.4}.

Let $U=L(-1)V=L(-1)V_+$. Define $P:U\ra V$ by setting
\begin{equation}\mlabel{2.18'}
    P(u):=L(-1)^{-1}u,
\end{equation}
for all $u\in L(-1)V_+$. Clearly, $P$ is well defined and $L(-1)P=\Id_{U}$. Then by Proposition \mref{prop2.14}, $P:U\ra V$ given by \meqref{2.18'} is a weak local RBO on $U=L(-1)V_+$ of weight $0$, and it is homogeneous of degree $-1$ and translation invariant. Note that $P$ is not ordinary since $P(U)=V_+$ is not a vertex Leibniz subalgebra of $V$.
\end{example}

\begin{example}\label{ex2.16}
Let $V=M_{\hat{\h}}(k,0)$ be the level $k\neq 0$ Heisenberg VOA of rank $r$ (cf. \cite{FLM}, see also \cite{FZ}). Recall that $\h$ is an $r$-dimensional vertex space, equipped with a nondegenerate symmetric bilinear form $(\cdot|\cdot)$, and $M_{\hat{\h}}(k,0)$ is the Verma module over the infinite dimensional Heisenberg Lie algebra: $\hat{\h}=\h\otimes \C[t,t^{-1}]\op \C K$, with
\begin{equation}\mlabel{2.13}
[\al(m),\b(n)]=m(\al|\b)\delta_{m+n,0}K, \quad m,n \in \Z,
\end{equation}
where $\al(m)=\al\otimes t^m$. We have $\hat{\h}=\widehat{\h}_{\geq 0}\op \widehat{h}_{<0}$, where $\widehat{h}_{\geq 0}=\h\otimes \C[t]\op \C K$ and $\widehat{\h}_{<0}=\h\otimes t^{-1}\C[t^{-1}]$, and $M_{\hat{\h}}(k,0)=U(\hat{\h})\otimes_{U(\hat{\h}_{\geq 0})}\C\vac$, with $\al(n)\vac=0$ for all $n\geq 0$ and $\al\in \h$, and $K\vac=k\vac.$

In particular, $\al(0)\in \hat{\h}$ is a central element by
\eqref{2.13}, and $\al(0)u=0$ for all $u\in
M_{\hat{\h}}(k,0)$. Fix a nonzero element $\al\in \h$, and
consider the operator $d=\al(1):V\ra V$. For $u,v\in
V=M_{\hat{\h}}(k,0)$ and $n\in \Z$, we have
\begin{align*}
\al(1)(u_nv)&=u_n(\al(1)v)+[\al(1),u_n]v=u_n(\al(1)v)+\sum_{j\geq 0}\binom{1}{j} (\al(j)u)_{1+n-j}v
=u_n(\al(1)v)+(\al(1)u)_n v,
\end{align*}
since $\al(0)u=0$. Thus, $d=\al(1)$ is a weak $0$-derivation on $M_{\hat{\h}}(k,0)$. By \eqref{2.13}, it is also easy to see that $d=\al(1)$ acts as $=k\frac{\partial}{\partial \al(-1)}$ on $M_{\hat{\h}}(k,0)$.  Hence $\al(1)M_{\hat{\h}}(k,0)=M_{\hat{\h}}(k,0)$. Define a linear map $P: \al(1)V=V\ra V$ as follows.
\begin{equation}\label{2.14}
\begin{aligned}
P:=\frac{1}{k}\int (\cdot) d\al(-1)\vac : M_{\hat{\h}}(k,0)&\ra M_{\hat{\h}}(k,0), \\
h^1(-n_{1})\dots h^k(-n_k)\al(-1)^m\vac&\mapsto \frac{1}{k(m+1)}h^1(-n_{1})\dots h^k(-n_k)\al(-1)^{m+1}\vac,
\end{aligned}
\end{equation}
where $S=\{\al=\al_1,\al_2,\dots \al_r\}$ is a basis of $\h$, and $h^1,\dots,h^k\in S$ are not equal to $\al$. Clearly, we have $dP=\Id_{V}$, and so $P:V\ra V$ is a weak global RBO on $M_{\hat{\h}}(k,0)$ of weight $0$ by Proposition \ref{prop2.14}. The operator $P$ is also homogeneous of degree $1$. However, it is not an ordinary RBO since $P(V)=\al(-1)M_{\hat{\h}}(k,0)$ is not a vertex Leibniz subalgebra.
\end{example}

\begin{example}
Let $V=V_{\hat{\g}}(k,0)$ be the level $k$ vacuum module VOA associated with $\g=\mathfrak{sl}(2,\C)=\C e\op \C h\op \C f$ (cf. \cite{FZ}). Let $V_{\hat{\g}}(k,0)=U(\hat{\g})\otimes _{U(\hat{\g}_{\geq 0})}\C\vac$ be the Weyl vacuum module over the affine Lie algebra $\hat{\g}=\g\otimes \C[t,t^{-1}]\op \C K$.

Since $h=h(-1)\vac\in V_1$, the map $d=o(h)=h(0): V\ra V$ is a $0$-derivation of $V$ (cf. \cite{DG1}). Moreover, $V_{\hat{\g}}(k,0)$ is a sum of $h(0)$-eigenspaces (\cite{DW}):
$V_{\hat{\g}}(k,0)=\bigoplus_{\la\in 2\Z} V_{\hat{\g}}(k,0)(\la),$
where $V_{\hat{\g}}(k,0)(\la)=\{v\in V_{\hat{\g}}(k,0)\,|\,h(0)v=\la v\}$ for all $\la\in 2\Z$.

Let $U$ be the sum of nonzero eigenspaces of $h(0)$: $U=\bigoplus_{\la\in 2\Z\bs\{0\}} V_{\hat{\g}}(k,0)(\la)$, and let $P: U\ra V$ be given by
$P(u)=\frac{1}{\la} u$
for all $u\in V_{\hat{\g}}(k,0)(\la)$, with $\la\neq 0$. Then $dP=\Id_{U}$, and so $P:U\ra V$ is a weak local RBO on $U$ of weight $0$. Moreover, $P$ is homogeneous of weight $0$. However, it is not ordinary since $P(U)=U$ is not a subalgebra.

Let $d_1:=e^{h(0)}-1: V\ra V$. Then $d_1$ is a $1$-derivation by Proposition \ref{prop2.13}. Let $P_{1}:U\ra V$ be given by $P_{1}(u)=\frac{1}{e^\la-1}u$, for all $u\in V_{\hat{\g}}(k,0)(\la)$, with $\la\neq 0$. Then $d_{1}P_{1}=\Id_{U},$ and by Proposition \ref{prop2.14}, $P_{1}:U\ra V$ is a weak local RBO of weight $1$.
\end{example}

\subsection{Properties and further examples of RBVAs}

The next theorem generalizes a basic property of RBOs (see
\cite[Theorem 1.1.13]{G}) and provides a systematic way to produce
examples of RBVAs.

\begin{thm}\mlabel{thm2.11}
Let $(V,Y,\vac)$ be a vertex algebra and $P:V\ra V$ a linear map. Then $P$ is an idempotent RBO of weight $-1$, if and only if $V$ admits a decomposition: $V=V^{1}\op V^{2}$ into a direct sum of vertex Leibniz subalgebras $V^{1}$ and $V^{2}$, and $P :V\ra V^{1}$ is the projection map onto $V^{1}$:
$$P(a^{1}+a^{2})=a^{1},\quad a^{1}\in V^{1}, a^{2}\in V^{2}.$$
In particular, $V^{1}=P(V)$ and $V^{2}=(\Id- P)(V)$.
\end{thm}
\begin{proof}
Let $P:V\ra V$ be an idempotent RBO of weight $-1$. Then the idempotency gives the direct sum $V=P(V)\oplus (\Id-P)(V)$. Further, $V^{1}=P(V)\subseteq V$ is closed under the vertex operator $Y$, since by \meqref{2.7}, we have
$$(Pa)_{n}(Pb)=P(a_{n}P(b)+P(a)_{n}b-a_{n}b)\in P(V), \quad a,b\in V.$$
Hence $(V^{1},Y|_{V^{1}})$ is a vertex Leibniz subalgebra of $V$. By Proposition~\mref{prop2.9}, $V^{2}=(\Id-P)(V)$ is also a vertex Leibniz subalgebra. 
\delete{
Since $P^{2}=P$ by the assumption, we have $V=V^{1}\op V^{2}$. Moreover, for $a\in V$, we have
$$a=P(a)+(1-P)(a)=a^{1}+a^{2},$$
where $a^{1}=P(a)$ and $a^{2}=(1-P)(a)$. Note that $a^{1}$ and $a^{2}$ are unique as the sum is direct. Then $P(a^{1}+a^{2})=P(a)=a^{1}$ is the projection onto $V^{1}$.
}

Conversely, suppose that $V$ has a decomposition $V=V^{1}\op V^{2}$ into vertex Leibniz subalgebras, and $P:V\ra V^{1}$ is the projection. Then $P$ is idempotent. Further, for $a=a^{1}+a^2$ and $b=b^1+b^2$ in $V$, with $a^i,b^i\in V^i$ for $i=1,2$, we have
\begin{align*}
P(a)_{n}P(b)&=a^1_{n}b^1,\quad P((Pa)_nb)=P(a^{1}_{n}b^1+a^1_{n}b^2)=a^{1}_{n}b^1+P(a^1_{n}b^2),\\
P(a_nP(b))&=P(a^1_{n}b^1+a^2_{n}b^1)=a^1_{n}b^1+ P(a^2_{n}b^1),\\
P(a_nb)&=P(a^1_{n}b^1+a^1_{n}b^2+a^2_{n}b^1+a^2_{n}b^2)=a^1_{n}b^1+P(a^1_{n}b^2)+P(a^2_{n}b^1).
\end{align*}
It follows that $P(a)_{n}P(b)=  P((Pa)_nb)+ P(a_nP(b))-P(a_nb)$, for all $a,b\in V$. Thus $P:V\ra V$ is an idempotent RBO of weight $-1$.
\end{proof}

\begin{example}\mlabel{ex2.12}
Let $V=V_{L}$ be the lattice VOA (cf. \mcite{FLM}) associated with the rank one positive definite even lattice $L=\Z \al$, with $(\al|\al)=2N$ for some $N\in \Z_{>0}$. Recall that
$V_L=M_{\hat{\h}}(1,0)\otimes \C^\epsilon[L],$ where $\h=\C\al$, and $\epsilon: L\times L\ra \{\pm 1\}$ is a $2$-cocycle. The vertex operators are given by
\begin{align*}
&Y(\al(-1)\vac,z)=\al(z)=\sum_{n\in \Z} \al(n) z^{-n-1},\\
&Y(e^{m\al},z)=E^{-}(-m\al, z)E^{+}(-m\al,z)e_{m\al} z^{m\al},\\
&Y(\al(-n_{1}-1)\dots \al(-n_{k}-1)e^{m\al},z)={\tiny\begin{matrix}\circ \\\circ\end{matrix}}(\partial_{z}^{(n_1)}\al(z))\dots (\partial_{z}^{(n_k)}\al(z))Y(e^{m\al},z){\tiny\begin{matrix}\circ \\\circ\end{matrix}},
\end{align*}
where $ E^{\pm}(\al,z)=\exp\left(\sum_{n\in \Z_{\pm}}\frac{\al(n)}{n}z^{-n}\right)$, and $\partial_z^{(n)}=\frac{1}{n!} \frac{d^n}{dz^n}. $ Also recall that $V_L$ has the following decomposition  as a module over the Heisenberg VOA $M_{\hat{\h}}(1,0)$:
        \begin{equation}\mlabel{2.19}
            V_{\Z\al}=\bigoplus_{m\in \Z} M_{\hat{\h}}(1,m\al).\end{equation}
        Note that for $m,n\in \Z_{<0}$, we have
        \begin{align*}
            Y(e^{m\al},z)e^{n\al}&=E^{-}(-m\al, z)E^{+}(-m\al,z)e_{m\al} z^{m\al}(e^{n\al})\\
            &=E^{-}(-m\al, z)E^{+}(-m\al,z)\epsilon(m\al,n\al) z^{2Nmn} e^{(m+n)\al},
        \end{align*}
which is contained in $M_{\hat{\h}}(1,(m+n)\al)((z))$, with $m+n\in \Z_{<0}$, in view of the decomposition \meqref{2.19}. Then it follows that
$$V^1:=\bigoplus_{m\in \Z_{\geq 0}} M_{\hat{\h}}(1,m\al)\quad \mathrm{and}\quad V^2:=\bigoplus_{m\in \Z_{<0}} M_{\hat{\h}}(1,m\al)$$
are vertex Leibniz subalgebras of $V_{\Z\al}$, and $V_{\Z\al}=V^1\op V^2$. Note that $\vac\in M_{\hat{\h}}(1,0)\subset V^1$. Then by Theorem \mref{thm2.11}, the projection $P:V_{\Z\al}\ra V^1$ is a level-preserving idempotent RBO of weight $-1$. This construction can also be generalized to the higher rank case. Let $L=\Z \al_{1}\op \dots \op\Z\al_n$ be a positive-definite even lattice of rank $n$, with an orthonormal basis  $\{\al_1,\dots \al_n\}$. Consider the following additive subgroups of $L$:
    \begin{equation}\mlabel{2.12'}
L^1:=\Z\al_1\op\dots \Z\al_{n-1}\op \Z_{\geq 0} \al_n\quad \mathrm{and}\quad L^2:=\Z\al_1\op\dots \Z\al_{n-1}\op \Z_{<0} \al_n.
        \end{equation}
Then we have $L=L^1\cup L^2$ and $L^1\cap L^2=\emptyset$. Let
        \begin{equation}\mlabel{2.13'}
V^1:=\bigoplus_{\al\in L^1}M_{\hat{\h}}(1,\al)\quad \mathrm{and}\quad V^2:=\bigoplus_{\b\in L^2} M_{\hat{\h}}(1,\b).
        \end{equation}
Then for a similar reason as the rank one case, $V^1$ and $V^2$ are vertex Leibniz subalgebras of $V_L$, and so the projection map $P:V_{L}\ra V^1$ along $V^2$ is a level-preserving RBO of weight $-1$. Furthermore, recall that the Virasoro element of $V_L$ is $\om =\frac{1}{2} \sum_{i=1}^n \al_{i}^2(-1)\vac$. In particular, we have $L(-1)M_{\hat{\h}}(1,\al) \ssq M_{\hat{\h}}(1,\al)$ for $\al \in L$. Thus $L(-1)V^i\ssq V^i$ for $i=1,2$, and so the RBO $P:V_{L}\ra V^1$ is translation invariant: $PL(-1)=L(-1)P$.
    \end{example}

\delete{\begin{example}\mlabel{ex 2.14}
Let $(V,Y,\vac,\om)$ be a VOA, and let $(W,Y_{W})$ be a weak $V$-module. It is observed in \mcite{L1} (see also \mcite{FHL}) that $V\op W$ carries the structure of a vertex algebra, with the vertex operator given by
\begin{equation}\mlabel{2.11}
Y_{V\op W}(a+v,z)(b+w):=(Y(a,z)b)+(Y_{W}(a,z)w+Y_{WV}^{W}(v,z)b),
\quad a,b\in V, v,w\in W,
\end{equation}
where $Y_{WV}^{W}$ is defined by the skew-symmetry formula:
    \begin{equation}\mlabel{2.12}
Y_{WV}^{W}(v,z)b=e^{zL(-1) }Y_{W}(b,-z)v.
        \end{equation}
        We can think of $V\op W$ as the semi-direct product $V\rtimes W$ of the VOA $V$ with the weak-module $W$. Since $Y_{V\op W}(v,z)w=0$ for $v,w\in W$, it follows that $(V,Y)$ and $(W,Y_{V\op W}|_{W})$ are vertex Leibniz subalgebras of $V\rtimes  W$. Then by Theorem \mref{thm2.11},
        \begin{equation}\mlabel{2.13b}P:V\rtimes W\ra V, a+v\mapsto a,\end{equation}
        is an RBO of weight $-1$ on the vertex algebra $V\rtimes W$.

        If $W$ only has integral weights, then $(V\rtimes W,Y_{V\op W}, \vac,\om)$ is a VOA (cf. \mcite{L1}). Then $P$ in \meqref{2.13b} is a level-preserving RBO of weight $-1$.
        This example will be used in the discussion of the next section.
    \end{example}}

There is another general construction of ordinary RBOs on vertex algebras. First, we have the following lemma whose proof is straightforward.
\begin{lm}\mlabel{lm3.19}
Let $(V,Y,\vac)$  be a vertex algebra, and $(A,\cdot,1)$ be a commutative unital associative algebra. Let $\hat{V}:=V\o _{\C} A$. Extend $Y:V\ra \End(V)[[z,z^{-1}]]$ to
\begin{equation}\mlabel{3.28}
\hat{Y}:\hat{V}\ra \End(\hat{V})[[z,z^{-1}]],\quad \hat{Y}(a\o f,z)(b\o g):=Y(a,z)b\o f\cdot g,\quad a,b\in V, f,g\in A.
\end{equation}
Then $(\hat{V},\hat{Y},\vac\o 1)$ is a vertex algebra, and the translation operator $\hat{D}:\hat{V}\ra \hat{V}$ is given by $\hat{D}(a\o f)=a_{-2}\vac\o f=D(a)\o f$.
\end{lm}

\begin{prop}\mlabel{prop3.20}
    With the notations as above, let $(V,Y,\vac)$  be a vertex algebra, and $(A,\cdot,1,P)$ be a commutative unital Rota-Baxter algebra. Let
    \begin{equation}\mlabel{3.29}
\hat{P}: \hat{V}\ra \hat{V},\quad   \hat{P}(a\o f):=a\o P(f)
,\quad a\o f\in \hat{V}.
    \end{equation}
Then $(\hat{V}, \hat{Y}, \vac\o 1, \hat{P})$  is an Rota-Baxter vertex algebra of weight $\la$.
    \end{prop}
\begin{proof}
By \meqref{3.28} and \meqref{3.29}, we have
    \begin{align*}
    \hat{P}\left(\hat{Y}(\hat{P}(a\o f), z)b\o g\right)&=\hat{P}\left( Y(a,z)b\o P(f)\cdot g\right)=Y(a,z)b\o P(P(f)\cdot g),\\
    \hat{P}\left(\hat{Y}(a\o f, z)\hat{P}(b\o g)\right)&=\hat{P}\left(Y(a,z)b\o f\cdot P(g)\right)=Y(a,z)b\o P(f\cdot P(g)),\\
    \la \hat{P}\left(\hat{Y}(a\o f,z) b\o g\right)&=\la \hat{P}\left(Y(a,z)b\o f\cdot g\right) =Y(a,z)b\o \la P(f\cdot g),\\
\hat{Y}\left(\hat{P}(a\o f),z\right)\hat{P}(b\o g)&=\hat{Y}\left(a\o P(f),z \right)b\o P(g)=Y(a,z)b\o P(f)\cdot P(g),
    \end{align*}
    for all $a,b\in V$, and $f,g\in A$. Since $P(f)\cdot P(g)=P(P(f)\cdot g)+P(f\cdot P(g))+\la P(f\cdot g)$, it follows that $\hat{P}:\hat{V}\ra \hat{V}$ is a RBO of the vertex algebra $\hat{V}$ of weight $\la$.
    \end{proof}
\begin{example}
    Recall that the Laurent polynomial ring has a decomposition into sub-algebras: $\C((t))= t^{-1}\C[t^{-1}]\op \C[[t]]$, and the projection operator $P:\C((t))\ra t^{-1}\C[t^{-1}]$ is a RBO of weight $-1$ (see \cite[Example 1.1.10]{G}).

    Then by Proposition~\mref{prop3.20}, for any vertex algebra $(V,Y,\vac)$, the projection map
    $$\hat{P}: V\o \C((t))\ra V\o t^{-1}\C[t^{-1}],\quad \hat{P}(a\o f(t))=a\o P(f(t)),\quad  a\in V, f(t)\in \C((t)),$$
    is a RBO of weight $-1$ on the vertex algebra $(\hat{V}=V\o \C((t)),\hat{Y},\vac\o 1)$ given by Lemma~\mref{lm3.19}, and $\hat{P}$ is translation invariant. In fact, this conclusion also follows from Theorem~\mref{thm2.11} since it is clear that $V\o t^{-1}\C[t^{-1}]$ and $V\o \C[[t]]$ are vertex Leibniz sub-algebras of $V\o \C((t))$.

We can similarly build examples of RBVAs from some of
the examples of commutative unital Rota-Baxter algebras
such as in~\mcite{G}.
    \end{example}

%\cm{we should unify ``commutative associative unital" or ``unital
%commutative"?} \jq{I also think the first one is better}
%\lir{Use the first one?}

Although it is generally not easy to classify all the Rota-Baxter operators on a given VOA, we can determine all the homogeneous Rota-Baxter operators of non-positive degree on certain CFT-type VOAs (as defined in Definition~\mref{df2.2}).

\begin{lm}\mlabel{lm2.14}
Let $(V,Y,\vac,\om)$ be a VOA of CFT-type, and $P:V\ra V$ be a homogeneous RBO of weight $\lambda\neq 0$
and degree $N\leq 0$. Then we have $P(\vac)=0$ or $P(\vac)=-\la \vac$. Furthermore, $P^{2}+\la P=0$.
    \end{lm}
\begin{proof}
Since $V_n=0$ for $n<0$, $V_0=\C\vac$, and $PV_{0}\ssq V_{N}$ for some $N\leq 0$, we have $P(\vac)=\mu \vac$ for some $\mu \in \C$. Recall that $\vac_{-1}\vac=\vac$. Then by \meqref{2.7}, we have
\begin{align*}
&P(\vac)_{-1}P(\vac)=P(P(\vac)_{-1}\vac)+P(\vac_{-1}P(\vac))+\la P(\vac_{-1}\vac),&
\end{align*}
yielding $\mu^{2}\vac=\mu^{2}\vac +\mu^{2}\vac+ \la\mu \vac.$
Hence $\mu$ is either $0$ or $-\la$. i.e., $P(\vac)=0$ or $-\la \vac$. Furthermore, again by \meqref{2.7} we have
        \begin{equation}\mlabel{2.14b}
P(a)_{-1}P(\vac)=P(P(a)_{-1}\vac)+P(a_{-1}P(\vac))+\la P(a_{-1}\vac), \quad a\in V.
        \end{equation}
If $P(\vac)=0$, then \meqref{2.14b} becomes
        $0=P(P(a))+P(a_{-1}0)+\la P(a),$
        and so $P^{2}(a)+\la P(a)=0$ for all $a\in V$. On the other hand, if $P(\vac)=-\la \vac$, then
        $-\la P(a)=P(P(a))-\la P(a)+\la P(a),$
        which also implies $P^2(a)+\la P(a)=0$, for all $a\in V$. Therefore, $P^{2}+\la P=0$.
    \end{proof}

\begin{prop}\mlabel{prop2.16}
Let $(V,Y,\vac,\om)$ be a VOA of CFT-type, and $P:V\ra V$ be a homogeneous RBO of weight $\lambda\neq 0$ and degree $N\leq 0$. Then $V=V^{1}\op V^{2}$, where $V^{1}$ and $V^{2}$ are graded vertex Leibniz subalgebras of $V$, with $V_{n}=V^{1}_{n}\op V^{2}_{n}$ for each $n\in \N$, and
$$P: V\ra V^{1}, a^{1}+a^{2}\mapsto -\la a^{1}, \quad a^{i}\in V^{i}, i=1,2.$$
Moreover, we have $P(\vac)=0$ if and only if  $V^1_0=0$, and  $P(\vac)=-\la \vac$ if and only if $V^2_0=0$.
\end{prop}
\begin{proof}
Since $P^{2}+\la P=0$ by Lemma \mref{lm2.14}, and $\la\neq 0$ by assumption, the RBO $-P/\la$ is an idempotent. Then by Proposition~\mref{prop2.9} $-P/\la$ is an RBO on $V$ of weight $-1$. By Theorem \mref{thm2.11}, we have $V=V^{1}\op V^{2}$, where $V^{1}=(-P/\la)(V)=P(V)$ and $V^{2}=\ker (-P/\la)=\ker P$ are vertex Leibniz subalgebras, and $-P/\la$ is the projection
$$-\frac{P}{\la}: V\ra V^{1}, a^{1}+a^{2}\mapsto a^{1}.$$
Hence $P(a^{1}+a^{2})=-\la a^{1}$. Moreover, since $P(V_n)\subset V_{n+N}$ for all $n\in \N$, we have $V^1=PV=\bigoplus_{m=-N}^\infty P(V_m)=\bigoplus_{n=0}^\infty V^1_n$, and $V^2=\bigoplus_{m=-N}^\infty \ker (P|_{V_m})=\bigoplus_{n=0}^\infty V^2_n$, where $V^{i}_n, i=1,2,$ are eigenspaces of $L(0)$ of eigenvalue $n$, and $V_n=V^{1}_n\op V^2_n$ for each $n\in \N$. Now the last statement is also clear as $V^1_0\op V^2_0=V_0=\C\vac$.
\end{proof}
\begin{coro}
Let $V$ be the level one Heisenberg VOA $M_{\hat{\mathfrak{h}}}(1,0)$ associated with $\h=\C\al$ or the Virasoro VOA $L(c,0)$ (see \mcite{FZ}), and let $P:V\ra V$ be a homogeneous RBO of degree $N\leq 0$ and weight $\la\neq 0$. Then $P$ is either $0$ or $-\la \Id_{V}$.
\end{coro}
\begin{proof}
Let $V=M_{\hat{\mathfrak{h}}}(k,0)$ or $L(c,0)$. By Proposition \mref{prop2.16}, $V=V^1\op V^2$ for some graded vertex Leibniz subalgebras $V^1,V^2\subset V$. But $V$ is generated by a single homogeneous element: $V=M_{\hat{\mathfrak{h}}}(k,0)$ is generated as a vertex algebra by $\al(-1)\vac\in V_1$ and $V_1=\C \al(-1)\vac=V^1_1\op V^2_1$, and $V=L(c,0)$ is generated by $\om \in V_{2}=\C\om =V^1_2\op V^2_2$. Then the single generator $u$ of $V$ is contained in either $V^1$ or $V^2$ for both cases. If $u\in V^1$, then $V=V^1$ and $P=-\la \Id_V$; if $u\in V^2$ then $V=V^2$ and $P=0$.
    \end{proof}

    \begin{df}
        Let $(V,Y,\vac,P)$ be an RBVA of weight $\la$. Define a new linear operator $Y^{\star_P}:V\ra (\End V)[[z,z^{-1}]]$ as follows:
        \begin{equation}\mlabel{2.16}
            Y^{\star_P}(a,z)b:=Y(a,z)Pb+Y(Pa,z)b+\lambda Y(a,z)b, \quad a,b\in V.
        \end{equation}
    \end{df}
    Note that $Y^{\star_P} $ is a generalization of $Y_{P}$ in \meqref{2.9}.
    \begin{lm} \mlabel{lm2.19}
The operator $Y^{\star_P} $ satisfies the truncation property and the skew-symmetry \meqref{2.5}. If, in addition, $P$ is translation invariant ($DP=PD$), then $Y^{\star_P} $ also satisfies the $D$-derivative property \meqref{2.3} and the $D$-bracket derivative property \meqref{2.4}.
    \end{lm}
    \begin{proof}
Given $a,b\in V$, since $Y(a,z)Pb $, $Y(Pa,z)b$ and $\lambda Y(a,z)b$ are all truncated from below, we have $Y^{\star_P}(a,z)b\in V((z))$. Moreover, by \meqref{2.16} and the skew-symmetry of $Y$, we have
$$  Y^{\star_P}(a,z)b=e^{zD}Y(Pb,-z)a+e^{zD}Y(b,-z)Pa+\lambda e^{zD}Y(b,-z)a=   e^{zD}Y^{\star_P}(b,-z)a.$$
Hence $Y^{\star_P}$ also satisfies the skew-symmetry. Now assume that $DP=PD$. By \meqref{2.16}, \meqref{2.3} and \meqref{2.4} for $Y$, we have
        \begin{align*}
Y^{\star_P}(Da,z)b&=Y(Da,z)Pb+Y(PDa,z)b+\lambda Y(Da,z)b\\
&=\frac{d}{dz} Y(a,z)Pb+Y(DPa,z)b+\la \frac{d}{dz} Y(a,z)b=\frac{d}{dz} Y^{\star_P}(a,z)b,\\
[D, Y^{\star_P}(a,z)]b&=DY(a,z)Pb-Y(a,z)PDb+[D,Y(Pa,z)]b+\la [D,Y(a,z)]b\\
&=[D,Y(a,z)]Pb+[D,Y(Pa,z)]b+\la [D,Y(a,z)]b=            \frac{d}{dz} Y^{\star_P}(a,z)b.
        \end{align*}
Hence $Y^{\star _P}$ satisfies the $D$-derivative and $D$-bracket derivative properties.
\end{proof}

The next theorem is the vertex algebra analog of the derived
product structure of RBOs, first discovered for Lie algebras by
Semenov-Tian-Shansky~\mcite{S}. See \cite[Theorem~1.1.17]{G} for
associative algebras. For vertex algebras, it shows that
$Y^{\star_P}$ gives a new structure of a vertex Leibniz
algebra (see Definition \mref{df2.10}) or a vertex
algebra without vacuum (see Definition \mref{df2.5}) on an RBVA
$(V,Y,\vac,P)$.
    \begin{thm}\mlabel{thm2.18}
Let $(V,Y,\vac,P)$ be an RBVA of weight $\la$, and $Y^{\star_P}$ be given by \meqref{2.16}.  Then we have
        \begin{enumerate}
\item$P(Y^{\star_{P}}(a,z)b)=Y(Pa,z)Pb$, for all $a,b\in V$.
\mlabel{it:double1}
\item $(V,Y^{\star_{P}})$ is a vertex Leibniz algebra. If, furthermore, $P$ is translation invariant, then $(V,Y^{\star_{P}}, D)$ is a vertex algebra without vacuum.
\mlabel{it:double2}
\item $P$ is an RBO of weight $\la $ on the vertex Leibniz algebra $(V,Y^{\star_P})$.
\mlabel{it:double3}
        \end{enumerate}
    \end{thm}
\begin{proof}
(\mref{it:double1}) By \meqref{2.6} and \meqref{2.16}, we have
\begin{align*}
&Y(Pa,z)Pb=P(Y(a,z)Pb)+P(Y(Pa,z)b)+\lambda P(Y(a,z)b)\\
&=P(Y(a,z)Pb+Y(Pa,z)b+\lambda Y(a,z)b)\\
&=P(Y^{\star_{P}}(a,z)b).
\end{align*}

\noindent
(\mref{it:double2}) By Lemma \mref{lm2.19}, to show that $(V,Y^{\star_{P}})$ is a vertex Leibniz algebra, we only need to show that $Y^{\star_P}$ satisfies the Jacobi identity, or equivalently, the weak commutativity and weak associativity, in view of Theorem \mref{thm2.3}.

We only prove the weak commutativity. The proof of the weak associativity is similar.  Given $a,b,c\in V$, we need to find some $N\in \N$ (depending on $a$ and $b$) such that
        \begin{equation}\mlabel{2.18}
(z_1-z_2)^NY^{\star_{P}}(a,z_1)Y^{\star_{P}}(b,z_2)c=(z_1-z_2)^NY^{\star_{P}}(b,z_2)Y^{\star_{P}}(a,z_1)c.
        \end{equation}
We again apply \meqref{2.16} to expand $Y^{\star_{P}}(a,z_1)Y^{\star_{P}}(b,z_2)c$ and $Y^{\star_{P}}(b,z_2)Y^{\star_{P}}(a,z_1)c$ as follows.
        \begin{align*}
&Y^{\star_{P}}(a,z_1)Y^{\star_{P}}(b,z_2)c\\
%&=Y^{\star_{P}}(a,z_1)(Y(b,z_2)Pc+Y(Pb,z_2)c+\lambda Y(b,z_2)c)\\
&=Y(a,z_1)P(Y(b,z_2)Pc)+Y(Pa,z_1)Y(b,z_2)Pc+\lambda Y(a,z_1)Y(b,z_2)Pc\\
&\ \ +Y(a,z_1)P(Y(Pb,z_2)c)+Y(Pa,z_1)Y(Pb,z_2)c+\lambda Y(a,z_1)Y(Pb,z_2)c\\
&\ \ +\lambda Y(a,z_1)P(Y(b,z_2)c)+\lambda Y(Pa,z_1)Y(b,z_2)c+\lambda ^2Y(a,z_1)Y(b,z_2)c;\\
&Y^{\star_{P}}(b,z_2)Y^{\star_{P}}(a,z_1)\\
%&=Y^{\star_{P}}(b,z_2)(Y(a,z_1)Pc+Y(Pa,z_1)c+\lambda Y(a,z_1)c)\\
&=Y(b,z_2)P(Y(a,z_1)Pc)+Y(Pb,z_2)Y(a,z_1)Pc+\lambda Y(b,z_2)Y(a,z_1)Pc\\
            &\ \ +Y(b,z_2)P(Y(Pa,z_1)c)+Y(Pb,z_2)Y(Pa,z_1)c+\lambda Y(b,z_2)Y(Pa,z_1)c\\
            &\ \ +\lambda Y(b,z_2)P(Y(a,z_1)c)+\lambda Y(Pb,z_2)Y(a,z_1)c+\lambda ^2Y(b,z_2)Y(a,z_1)c.
        \end{align*}
        By \meqref {2.6} again, we have
        \begin{align*}
            &Y(a,z_1)P(Y(b,z_2)Pc)+Y(a,z_1)P(Y(Pb,z_2)c)+\lambda Y(a,z_1)P(Y(b,z_2)c)\\
            &=Y(a,z_1)Y(Pb,z_2)Pc,\\
            &Y(b,z_2)P(Y(a,z_1)Pc)+Y(b,z_2)P(Y(Pa,z_1)c)+\lambda Y(b,z_2)P(Y(a,z_1)c)\\
            &=Y(P(a),z_1)Y(b,z_2)Pc.
        \end{align*}
We can find a common natural number $N\in \N$, depending on $a$ and $b$, that ensures all the following weak commutativities of $Y$:
\begin{align*}
(z_1-z_2)^NY(a,z_1)Y(Pb,z_2)Pc&=(z_1-z_2)^NY(Pb,z_2)Y(a,z_2)Pc,\\
(z_1-z_2)^NY(b,z_2)Y(Pa,z_1)Pc&=(z_1-z_2)^NY(P(a),z_1)Y(b,z_2)Pc,\\
        (z_1-z_2)^NY(a,z_1)Y(b,z_2)Pc&=(z_1-z_2)^NY(b,z_2)Y(a,z_1)Pc,\\
        (z_1-z_2)^NY(Pa,z_1)Y(Pb,z_1)c&=(z_1-z_2)^NY(Pb,z_2)Y(Pa,z_1)c,\\
        (z_1-z_2)^NY(a,z_1)Y(Pb,z_2)c&=(z_1-z_2)^NY(Pb,z_2)Y(a,z_1)c,\\
        (z_1-z_2)^NY(Pa,z_1)Y(b,z_2c&=(z_1-z_2)^NY(b,z_2)Y(Pa,z_1)c,\\
        (z_1-z_2)^NY(a,z_1)Y(b,z_2)c&=(z_1-z_2)^NY(b,z_2)Y(a,z_1)c.
        \end{align*}
        This shows \meqref{2.18} by comparing the expansions. Thus, $(V,Y^{\star_P})$ is a vertex Leibniz algebra. If the RBO $P$ is also translation invariant, then by Lemma \mref{lm2.19}, $(V,Y^{\star_P},D)$ is a vertex algebra without vacuum.

        (\mref{it:double3}) By \meqref{2.16} and \meqref{it:double1},
we have
        \begin{align*}
            Y^{\star_{P}}(Pa,z)Pb&=Y(Pa,z)P(Pb)+Y(P(Pa),z)P(b)+\la Y(Pa,z)Pb\\
            &=P(Y^{\star_{P}}(a,z)Pb)+P(Y^{\star_{P}}(Pa,z)b)+\la P(Y^{\star_{P}}(a,z)b).
        \end{align*}
        By Definition \mref{df2.10}, $P$ is an RBO of weight $\la$ on the vertex Leibniz algebra $(V,Y^{\star_{P}})$.
    \end{proof}

\delete{
We now give the following analog of the additive decomposition of Atkinson~\mcite{At,G}.
 \begin{thm}{\bf (Additive decomposition)}
    Let $(V,Y,1)$ be a vertex algebra over the complex number field $\C$. For $m\in \Z$, a linear operator $P:V\ra V$ is an m-ordinary RBO of weight $\lambda $ on $V$ if and only if for $a,b\in V$, there exists $c\in V$ such that $P(a)_mP(b)=P(c)$ and $\tilde{P}(a)_m\tilde{P}(b)=-\tilde{P}(c)$, where $\tilde{P}=-\lambda id-P$ and $\lambda\neq 0$.
\end{thm}
\begin{proof}
    Let $P:V\ra V$ be an $m$-ordinary RBO of weight $\lambda $ on $V$. For $a,b\in V$, we have
    $$Pa_mPb=P(Pa_mb+a_mPb+\lambda a_mb).$$
    Take $c=Pa_mb+a_mPb+\lambda a_mb$. Then it is obvious that $P(a)_mP(b)=P(c)$. Moreover, from the proof of part two in Proposition~\mref{prop2.9},  we have $\tilde{P}(a)_m\tilde{P}(b)=-\tilde{P}(c)$. \par
    Conversely, for given $a,b\in V$, we will show that
    $$P(a)_mP(b)=P(Pa_mb)+P(a_mPb)+\lambda P(a_mb).$$
    Let $c\in V$ be such that $P(a)_mP(b)=P(c)$ and $\tilde{P}(a)_m\tilde{P}(b)=-\tilde{P}(c)$. Since $\tilde{P}=-\lambda \Id-P$, we have
    \begin{align*}
        -\lambda  c&=P(c)+\tilde{P}(c)=P(a)_mP(b)-\tilde{P}(a)_m\tilde{P}(b)\\&=P(a)_mP(b)-(\lambda^2a_mb+\lambda a_mP(b)+\lambda P(a)_mb+P(a)_mP(b))\\&=-\lambda (a_mP(b)+P(a)_mb+\lambda a_mb).
    \end{align*}
    Since $\lambda \neq 0$, we have
    $c=a_mP(b)+P(a)_mb+\lambda a_mb.$
    It follows that $ P(a)_mP(b)=P(c)=P(Pa_mb)+P(a_mPb)+\lambda P(a_mb),$ for all $m\in \Z$.
\end{proof}}

\begin{remark}
	If $\la=0$, then Theorem~\mref{thm2.18} gives an alternative proof to Proposition~\mref{prop3.3}. In particular, if $R=P:V\ra V$ is an RBO of weight $0$, or equivalently, satisfies the ``modified Yang-Baxter equation'', then $Y_R(a,z)b=Y(Ra,z)b+Y(a,z)Rb=Y^{\star_{P}}(a,z)b$ satisfies the Jacobi identity of VOAs, or equivalently, $R=P$ is an $R$-matrix for $V$. 
	\end{remark}

%\jq{I think this example might be needed to make better sense of Theorem~\mref{thm2.18}}
\begin{example}
Let $V=V_L$ be a lattice VOA, and let $P:V_{L}\ra V^1$ be the projection RBO in Example \mref{ex2.12}, where $L=L^1\sqcup L^2$ and $V_L=V^1\op V^2$ as in \meqref{2.12'} and \meqref{2.13'}. For $a=a^1+a^2$ and $b=b^1+b^2$ in $V_L$, where $a^i,b^i\in V^i$ for $i=1,2$, we have
$P(a^1)=a^1, P(b^1)=b^1,$ and $P(a^2)=P(b^2)=0.$
Then by \meqref{2.16}, with $\la=-1$, we have
\begin{align*}
Y^{\star_P}(a,z)b&=Y(a,z)Pb+Y(Pa,z)b- Y(a,z)b\\
&=Y(a^1+a^2)b^1+Y(a^1,z)(b^1+b^2)-Y(a^1+a^2,z)(b^1+b^2)\\
&=Y(a^1,z)b^1-Y(a^2,z)b^2.
\end{align*}
Since $P$ is translation invariant, by Theorem \mref{thm2.18}, $(V_L, Y^{\star_P},L(-1))$ is a vertex algebra without vacuum, with $Y^{\star_P}$ given by
\begin{equation}\mlabel{2.21}
    Y^{\star_P}(a,z)b=Y(a^1,z)b^1-Y(a^2,z)b^2.
\end{equation}
Note that the vacuum element $\vac$ of $V_L$ is contained in $M_{\hat{\h}}(1,0)\subset V^1$, and it cannot be the vacuum element of $(V_L, Y^{\star_P},L(-1))$, since $Y^{\star_P}(\vac,z)a^2=0$ for all $a^2\in V^2$ by \meqref{2.21}.
\end{example}

\section{Dendriform vertex algebras}
\mlabel{sec:dend}

%\cm{there is not a summary or an outline of this section at the
%beginning of this section.}

%\lir{Delete the first sentence which repeats the introduction. Then add a summary. Expand the relation with relative RBOs suggested in the introduction?}\jq{I added a brief summary.}

%The notion of a dendriform algebra was introduced by Loday~\cite{Lo} with motivation from the %periodity of algebraic $K$-theory and with broad connections, especially to RBOs. We extend this %notion and some related connections to vertex algebras.

In this section, we study the dendriform structure under the framework of vertex algebras, which has an intimate relation with the Rota-Baxter operators on vertex algebras. Since the usual dendriform axioms are the splitting axioms of associativity, while the key axiom of a vertex algebra is the Jacobi identity, we take the viewpoint that the Jacobi identity is the combination of weak associativity together with the $D$-bracket derivative property and skew-symmetry (cf. \mcite{HL,L3}), and define our dendriform structures gradually. In particular, we show that our definitions preserve the usual properties of a dendriform algebra, and they give rise to some new kinds of Jacobi identities (see Theorem~\mref{thm4.10}). 

%\cm{ some new kinds of Jacobi identities (see Theorem~\mref{thm4.10}).}

\subsection{Dendriform field and vertex algebras}

A \name{dendriform algebra} is a vector space $A$ over a field $k$, equipped with two binary operators $\prec$ and $\succ$, satisfying
    \begin{align}
(x\prec y)\prec z&=x\prec(y\prec z+y\succ z),\mlabel{4.30}\\
(x\succ y)\prec z&=x\succ (y\prec z),\mlabel{4.31}\\
(x\prec y+x\succ y)\succ z&=    x\succ (y\succ z),\mlabel{4.32} \quad x, y, z\in A.
    \end{align}
Given a dendriform algebra $(A, \prec,\succ)$, the product
    \begin{equation}\mlabel{4.32'}
x\star y:=x\prec y+x\succ y, \quad x, y\in A.
    \end{equation}
is associative. Furthermore, Rota-Baxter algebras give rise to dendriform algebras.
For instance, an Rota-Baxter algebra $(R,P)$ of weight $0$ defines a dendriform algebra $(R, \prec_{P},\succ_P)$~\mcite{A1}, where
        \begin{equation}\label{5.4}
            x\prec_P y=xP(y), \qquad x\succ_P y=P(x)y, \quad x, y\in R.
        \end{equation}

The associative analog of vertex algebras is the notion of field
algebras~\mcite{BK}, or the nonlocal vertex
algebras~\mcite{L3}. A \name{field algebra} $(V,Y,\vac,D)$ is a vector space $V$, equipped with a linear map $Y:V\ra (\End V)[[z,z^{-1}]]$, a distinguished vector
$\vac$, and a linear map $D :V\ra V$, satisfying the truncation
property, the vacuum and creation properties, the $D$(-bracket)
derivative properties \meqref{2.3} and \meqref{2.4}, and the weak
associativity \meqref{2.1}.

Since the axioms of a dendriform are extracted from the
associativity axiom, the weak associativity axiom \meqref{2.2} of
a field algebra is enough for our purpose. We introduce the
following notion, which can be viewed as a weaker version of both
field algebras and vertex Leibniz algebras. 
\begin{df}\mlabel{df4.2}
A \name{field Leibniz algebra} is a vector space $V$, equipped with a linear map $Y:V\ra (\End V)[[z,z^{-1}]]$, satisfying the truncation property and the weak associativity \meqref{2.2}. We denote a field Leibniz algebra by $(V,Y)$.

An \name{ordinary Rota-Baxter operator (ordinary RBO)} on a field Leibniz algebra $(V,Y)$ of weight $\la\in \C$ is a linear map $P:V\ra V$, satisfying the compatibility \meqref{2.7}.
\end{df}

    By a similar argument as the proof of Proposition~\mref{prop2.6} (cf. \mcite{HL}), together with Lemma 2.7 in \mcite{LTW}, we can derive the following relation between the field Leibniz algebras and the vertex algebras without vacuum.
\begin{prop}\mlabel{prop4.2}
Let $(V,Y)$ be a field Leibniz algebra. Suppose that there exists a linear map $D:V\ra V$ satisfying the $D$-bracket derivative property \meqref{2.4} and the skew-symmetry \meqref{2.5}. Then $(V,Y,D)$ is a vertex algebra without vacuum.
\end{prop}

Recall  the notions of vertex Leibniz algebras and vertex algebras without vacuum in
Definitions~\mref{df2.10} and \mref{df2.5}, respectively. It is clear that we have the following embedding of categories. 
\begin{equation}\mlabel{4.36'}
\text{ vertex\ alg.\ without\ vacuum}\subset \text{ vertex\ Leibniz\ alg.}\subset  \text{ field\ Leibniz\ alg.}
\end{equation}

Inspired by the axioms \meqref{4.30}--\meqref{4.32} of dendriform algebras and the weak-associativity \meqref{2.2}, we expect to decompose the vertex operator $Y(\cdot,z)$ into a sum of two operators: $Y(\cdot,z)=Y_{\prec}(\cdot,z)+Y_{\succ}(\cdot,z),$ whose properties are consistent with both the Rota-Baxter type axiom and the weak associativity axiom. Furthermore, we want the embedding of categories \meqref{4.36'} to be reflected as well. This leads us to the following definition.

\begin{df}\mlabel{df4.3}
Let $V$ be a vector space, and
        \begin{align*}
&Y_{\prec}(\cdot,z):  V\ra \Hom(V,V((z))),\  a\mapsto Y_{\prec}(a,z),\\
&Y_{\succ}(\cdot,z):  V\ra \Hom(V,V((z))),\  a\mapsto Y_{\succ}(a,z)
        \end{align*}
be two linear operators associated with a formal variable $z$. For simplicity, we denote $Y_{\prec}(\cdot,z)$ by $\cdot\prec_z \cdot$, and $Y_{\succ}(\cdot,z)$ by $\cdot\succ\cdot$, respectively, and write $$Y_\prec(a,z)b=a\prec_z b, \quad Y_\succ(a,z)b=a\succ_z b, \quad a, b\in V.$$
\begin{enumerate}
\item A triple $(V,\prec_{z},\succ_{z})$ is called a {\bf dendriform field algebra} if for $a,b,c\in V$, there exists some $N\in \N$ depending on $a$ and $c$, satisfying
\begin{align}
(z_0+z_2)^{N}   (a\prec_{z_0}b)\prec_{z_2}c &=  (z_0+z_2)^{N}a\prec_{z_0+z_2} (b\succ_{z_2}c+b\prec_{z_2}c), \mlabel{5.6}\\
(z_0+z_2)^{N}   (a\succ _{z_0}b)\prec_{z_{2}}c&=    (z_0+z_2)^{N}a\succ_{z_0+z_2}(b\prec_{z_2}c),\mlabel{5.7}\\
(z_0+z_2)^{N}   (a\succ_{z_0}b+a\prec_{z_0}b)\succ_{z_2} c&=    (z_0+z_2)^{N}a\succ_{z_0+z_2}(b\succ_{z_2}c).\mlabel{5.8}
\end{align}
    \item A triple $(V,\prec_z,\succ_z)$ is called a {\bf dendriform vertex Leibniz algebra} if it is a dendriform field algebra and satisfies the following additional conditions: for $a,b,c\in V$, there exists some $N\in \N$ depending on $a$ and $b$, such that
        \begin{align}
    (z_1-z_2)^N a\succ_{z_1}(b\prec_{z_2}c)&=(z_1-z_2)^N b\prec_{z_2}(a\succ_{z_1}c+a\prec_{z_1}c),\mlabel{4.43}\\
    (z_1-z_2)^N a\succ_{z_1}(b\succ_{z_2}c)&=(z_1-z_2)^N b\succ_{z_2}(a\succ_{z_1}c).\mlabel{4.44}
    \end{align}
\item  Let $D:V\ra V$ be a linear map. A quadruple $(V,\prec_{z},\succ_{z},D)$ is called a {\bf dendriform vertex algebra (without vacuum)} if $(V,\prec_{z},\succ_{z})$ is a dendriform field algebra, and $D$, $\prec_{z}$, and $\succ_{z}$ satisfy the following compatibility properties:
            \begin{align}
&e^{zD}(a\prec_{-z} b)=b\succ_{z} a \quad \mathrm{and}\quad     e^{zD}(a\succ_{-z} b)=b\prec_{z} a;\mlabel{4.38}\\
&D(a\prec_{z} b)-a\prec_{z} (Db)=\frac{d}{dz}(a\prec_{z} b)\quad \mathrm{and}\quad D(a\succ_{z} b)-a\succ_{z} (Db)=\frac{d}{dz} (a\succ_{z} b). \mlabel{4.39}
            \end{align}
        \end{enumerate}
\end{df}

\begin{remark}
    By the definition above, any dendriform vertex Leibniz algebra is automatically a dendriform field algebra. Later, we will see that the equations \meqref{5.6}--\meqref{5.8} are the underlying axioms of the weak associativity \meqref{2.2}, while the equations \meqref{4.43} and \meqref{4.44} are the underlying axioms of the weak commutativity \meqref{2.1}.

We will also show that the equations \meqref{5.6}--\meqref{5.8}
are equivalent with equations \meqref{4.43} and \meqref{4.44} if
there exists $D:V\ra V$ satisfying \meqref{4.38} and
\meqref{4.39} (Proposition~\mref{prop4.6}). Thus
any dendriform vertex algebra is also a dendriform vertex Leibniz
algebra (Corollary~\mref{cor4.7}), and we have a similar embedding of
categories as \meqref{4.36'}:
\begin{equation}
\text{dendriform\ vertex\ alg.}\subset\text{dendriform\ vertex\
Leibniz\ alg.}\subset \text{dendriform\ field\  alg}.
\end{equation}
\end{remark}

We can regard the equality \meqref{4.38} as the analog of skew-symmetry \meqref{2.5} satisfied by the partial operators $\prec_z$ and $\succ_z$. The equality \meqref{4.39} can be viewed as the $D$-bracket derivative property \meqref{2.4} for the partial operators. Similar to Lemma 2.7 in \mcite{LTW}, we have the following result.
\begin{lm}\mlabel{lm4.4}
    Let $V$ be a vector space, equipped with two operators $\prec_z,\succ_z: V\times V\ra V((z))$ and a linear operator $D:V\ra V$, satisfying the skew-symmetry \meqref{4.38}. Then the $D$-bracket derivative property \meqref{4.39} is equivalent to the following $D$-derivative property.
    \begin{equation}\mlabel{4.40'}
(Da)\prec_z b=\frac{d}{dz}(a\prec_z b),\quad \quad (Da)\succ_z b=\frac{d}{dz}(a\succ_z b),  \quad a, b\in V.\end{equation}
In particular, a dendriform vertex algebra $(V,\prec_z,\succ_z,D)$ can be defined as a dendriform field algebra $(V,\prec_z, \succ_z)$ that satisfies \meqref{4.38} and \meqref{4.40'}.
\end{lm}
\begin{proof}
Similar to the proof of Lemma 2.7 in \mcite{LTW}, for $a,b\in V$, we have
{\small
    \begin{align*}
    (Da)\prec_z b-\frac{d}{dz}(a\prec_zb)&=e^{zD}b\succ_{-z}Da-De^{zD}(b\succ_za)-e^{zD}\frac{d}{dz}(b\succ_{-z}a)\\
    &=e^{zD}\left(b\succ_{-z}Da-D(b\succ_{-z}a)-\frac{d}{dz}(b\succ_{-z}a)\right),\\
        (Da)\succ_z b-\frac{d}{dz}(a\succ_zb)&=e^{zD}b\prec_{-z}Da-De^{zD}(b\prec_za)-e^{zD}\frac{d}{dz}(b\prec_{-z}a)\\
    &=e^{zD}\left(b\prec_{-z}Da-D(b\prec_{-z}a)-\frac{d}{dz}(b\prec_{-z}a)\right).
    \end{align*}
}
    Thus, \meqref{4.39} is equivalent to \meqref{4.40'}.
    \end{proof}

The axioms of the dendriform algebra are closely related to the properties of Rota-Baxter operators. As their vertex algebra analogs, the following theorem shows that an RBVA of weight $\la$ gives rise to a dendriform field algebra, and an RBVA of weight $0$ gives rise to a dendriform vertex algebra.

    \begin{thm}\mlabel{thm5.3}
Let $(V,Y,\vac, P)$  be an RBVA of weight $\la$.
\begin{enumerate}
\item Let $\lambda$ be arbitrary. Then $(V,Y,\vac, P)$ defines a dendriform field algebra   $(V,\prec'_z, \succ'_z)$, where
\begin{equation}\mlabel{5.10}
a\prec'_z b=Y(a,z)P(b)+\lambda Y(a,z)b, \qquad a\succ'_z
b=Y(P(a),z)b,\quad a,b\in V.
\end{equation}
\mlabel{it:rbd2}
\item If $\la=0$, then $(V,Y,\vac, P)$ defines a dendriform vertex Leibniz algebra $(V,\prec_z, \succ_z)$, where
\begin{equation}\mlabel{5.9}
    a\prec_z b:=Y(a,z)P(b), \qquad a\succ_z b:=Y(P(a),z)b, \quad a, b\in V.
\end{equation}
If $P$ is also translation invariant, then $(V,\prec_{z},\succ_{z},D)$ is  a dendriform vertex algebra.
\mlabel{it:rbd1}
\end{enumerate}
    \end{thm}
    \begin{proof}
\meqref{it:rbd2}.
To verify equation \meqref{5.6}, we have
        \begin{align*}
            (z_0+z_2)^{N}   (a\prec'_{z_0}b)\prec'_{z_2}c &=(z_0+z_2)^N(Y(Y(a,z_0)P(b)+\lambda Y(a,z_0)b,z_2)P(c)\\
            &+\lambda Y(Y(a,z_0)P(b)+\lambda Y(a,z_0)b,z_2)c).
        \end{align*}
        On the other hand,
        \begin{align*}
            &(z_0+z_2)^{N}a\prec'_{z_0+z_2} (b\succ'_{z_2}c+b\prec'_{z_2}c)\\&=(z_0+z_2)^N(Y(a,z_0+z_2)P(Y(P(b),z_2)c+Y(b,z_2)P(c)+\lambda Y(b,z_2)c)\\
            &+\lambda Y(a,z_0+z_2)(Y(P(b),z_2)c+Y(b,z_2)P(c)+\lambda Y(b,z_2)c))\\
            & =(z_0+z_2)^N(Y(a,z_0+z_2)Y(P(b),z_2)P(c)\\
            &+\lambda Y(a,z_0+z_2)(Y(P(b),z_2)c+Y(b,z_2)P(c)+\lambda Y(b,z_2)c) ).
        \end{align*}
Take a common $N$ that ensures the weak associativity \meqref{2.2} for $(a, P(b), P(c))$, $(a,b, P(c))$, $(a, P(b), c)$, and $(a, b,c)$ at the time. Then equation \meqref{5.6} holds. \meqref{5.7} and \meqref{5.8} can be proved similarly, we omit the details.

\delete{For equation \meqref{5.7}, we have
\begin{align*}(z_0+z_2)^{N} (a\succ' _{z_0}b)\prec'_{z_{2}}c&=(z_0+z_2)^N(Y(Y(P(a),z_0)b,z_2)P(c)+\lambda Y(Y(P(a),z_0)b,z_2)c),\\
(z_0+z_2)^{N}a\succ'_{z_0+z_2}(b\prec'_{z_2}c)&=(z_0+z_2)^NY(P(a),z_0+z_2)(Y(b,z_2)P(c)+\lambda Y(b,z_1)c).
        \end{align*}
We again take a common $N$ such that the weak associativities for $(P(a), b, P(c))$ and $(P(a), b, c)$ are all satisfied. Then equation \meqref{5.7} holds. Finally, for equation \meqref{5.8}, we have
        \begin{align*}
            (z_0+z_2)^{N}   (a\succ'_{z_0}b+a\prec'_{z_0}b)\succ'_{z_2}c
            &=(z_0+z_2)^NY(P(Y(P(a),z_0)b+Y(a,z_0)P(b)+\lambda Y(a,z_0)b),z_2)c\\
            &=(z_0+z_2)^NY(Y(P(a),z_0)P(b),z_2)c,\\
            (z_0+z_2)^{N}a\succ'_{z_0+z_2}(b\succ'_{z_2}c)&=(z_0+z_2)^NY(P(a),z_0+z_2)(Y(P(b),z_2)c).
        \end{align*}
Take a common $N$ such that the weak associativities for $(P(a), P(b),c)$ are satisfied. Then equation \meqref{5.8} holds. Note that the $N$ we chosen depends only on $a$ and $c$, in view of \meqref{2.2}, since $P$ is fixed. This proves \meqref{it:rbd2}.}
\smallskip

\noindent
\meqref{it:rbd1}. By taking $\la=0$, we see that $(V,\prec_z,\succ_z)$ given by \meqref{5.9} is a dendriform field algebra. Furthermore, by \meqref{5.9} and \meqref{2.7}, we have
\begin{align*}
 a\succ_{z_1} (b\prec_{z_2}c)&=a\succ_{z_1}(Y(b,z_2)P(c))=Y(P(a),z_1)Y(b,z_2)P(c),\\
 b\prec_{z_2}(a\succ_{z_1}c+a\prec_{z_1}c)&=Y(b,z_2)P(Y(P(a),z_1)c+Y(a,z_1)P(c))=Y(b,z_2)Y(P(a),z_1)P(c).
\end{align*}
Choose a common $N\in \N$ so that the weak commutativity \meqref{2.1} holds for $(P(a),b,P(c))$. Then we have \meqref{4.43}. By \meqref{5.9} and \meqref{2.7} again,
\begin{align*}
a\succ_{z_1}(b\succ_{z_2}c)&=a\succ_{z_1}(Y(P(b),z_2)c)=Y(P(a),z_1)Y(P(b),z_2)c,\\
b\succ_{z_2}(a\succ_{z_1}c)&=b\succ_{z_2}(Y(P(a),z_1)c)=Y(P(b),z_2)Y(P(a),z_1)c.
\end{align*}
Choose an $N\in \N$ so that the weak commutativity \meqref{2.1} holds for $(P(a),P(b),c)$. Then we have \meqref{4.44}. Note that the $N$ we have chosen depends only on $a$ and $b$ since $P$ is fixed. Thus, $(V,\prec_z,\succ_z)$ is a dendriform vertex Leibniz algebra.

Finally, if $P$ is translation invariant ($PD=DP$), then by \meqref{5.9} we have
        \begin{align*}
            e^{zD}(a\prec_{-z}b )&=e^{zD}Y(a,-z)P(b)=Y(P(b),z)a=b\succ_{z} a,\\
            e^{zD}(a\succ_{-z}b )&=e^{zD}Y(P(a),-z)b=Y(b,z)P(a)=b\prec_{z} a.
        \end{align*}
Moreover, the $D$-derivative and bracket derivative properties \meqref{2.3} and \meqref{2.4} yield
        \begin{align*}
            D(a\prec_z b)-a\prec_{z} (Db)&=DY(a,z)P(b)-Y(a,z)P(Db)=[D,Y(a,z)]P(b)\\
            &=\frac{d}{dz} Y(a,z)P(b)=\frac{d}{dz} (a\prec_z b),\\
            D(a\succ_{z} b)-a\succ_{z} (Db)&=DY(P(a),z)b-Y(P(a),z)Db=[D,Y(P(a),z)]b \\
            &=\frac{d}{dz}(Y(P(a),z)b)=\frac{d}{dz} (a\succ_{z} b).
        \end{align*}
Hence $(V,\prec_z,\succ_z, D)$ is a dendriform vertex algebra, in view of \meqref{4.38} and \meqref{4.39}.
    \end{proof}

Dendriform field algebras, vertex Leibniz algebras, and vertex algebras also give rise to field Leibniz algebras (see Definition \mref{df4.2}),  vertex Leibniz algebras (see Definition~\mref{df2.10}), and vertex algebras without vacuum (see Definition \mref{df2.5}), respectively.

\begin{thm}\mlabel{thm4.5}
Let $V$ be a vector space, equipped with two linear maps  $\prec_z, \succ_z: V\ra \Hom(V,V((z)))$ as in Definition~\mref{df4.3}. Define $Y:V\ra (\End V)[[z,z^{-1}]]$ by
\begin{equation}\mlabel{5.11}
Y(a,z) b:=a\prec_zb+a\succ_z b, \quad  a, b\in V.
\end{equation}
Then we have the following properties, with $Y$ given by \meqref{5.11}$:$
\begin{enumerate}
\item If $(V,\prec_z,\succ_z)$ is a dendriform field algebra, then $(V,Y)$ is a field Leibniz algebra.
\item If $(V,\prec_z,\succ_z)$ is a dendriform vertex Leibniz algebra, then $(V,Y)$ is a vertex Leibniz algebra.
 \item If there exists a linear map $D:V\ra V$ such that $(V, \prec_z, \succ_z,D )$ is a dendriform vertex algebra, then $(V,Y,D)$ is a vertex algebra without vacuum.
\end{enumerate}
\end{thm}
\begin{proof}
Clearly, $Y$ defined by \meqref{5.11} satisfies the truncation
property, in view of Definition~\mref{df4.3}. Let
$(V,\prec_z,\succ_z)$ be a dendriform field algebra. We claim
that $Y$ satisfies the weak associativity \meqref{2.2}. Indeed,
for all $a,b,c \in V$, we have
        \begin{align*}
(z_0+z_2)^N Y(&Y(a,z_0)b,z_2)c=(z_0+z_2)^kY(a\prec _{z_0}b+a\succ_{z_0}b,z_2)c\\
&=(z_0+z_2)^N(a\prec_{z_0}b+a\succ_{z_0}b)\prec_{z_2}c+(a\prec_{z_0}b+a\succ_{z_0}b)\succ_{z_2}c\\
&=(z_0+z_2)^N((a\prec_{z_0}b)\prec_{z_2}c+(a\succ_{z_0}b)\prec_{z_2}c+(a\prec_{z_0}b+a\succ_{z_0}b)\succ_{z_2}c),\\
(z_0+z_2)^NY(&a,z_0+z_2)Y(b,z_2)c=(z_0+z_2)^kY(a,z_0+z_2)(b\prec_{z_2}c+b\succ_{z_2}c)\\
&=(z_0+z_2)^N(a\prec_{z_0+z_2}(b\prec_{z_2}c+b\succ_{z_2}c)+a\succ_{z_0+z_2}(b\prec_{z_2}c+b\succ_{z_2}c))\\
&=(z_0+z_2)^N(a\prec_{z_0+z_2}(b\prec_{z_2}c+b\succ_{z_2}c)+a\succ_{z_0+z_2}(b\prec_{z_2}c)+a\succ_{z_0+z_2}(b\succ_{z_2}c)).
\end{align*}
We take a common $N>0$ depending on $a$ and $c$, such that equations \meqref{5.6}, \meqref{5.7}, and \meqref{5.8} are satisfied at the same time. Then we have the weak associativity \meqref{2.2}. Hence $(V,Y)$ is a field Leibniz algebra.

Let $(V,\prec_z,\succ_z)$ be a dendriform vertex Leibniz
algebra. Choose a $N\in \N$ depending on $a$ and $b$ such that
\meqref{4.43} and \meqref{4.44} hold simultaneously. Then
\begin{align*}
&(z_1-z_2)^NY(a,z_1)Y(b,z_2)c\\
&=(z_1-z_2)^N(a\prec_{z_1} (b\prec_{z_2}c+b\succ_{z_2}c)+a\succ_{z_1}(b\prec_{z_2}c)+a\succ_{z_1}(b\succ_{z_2}c))\\
&=(z_1-z_2)^N (b\succ_{z_2}(a\prec_{z_1}c))+(z_1-z_2)^N(b\prec_{z_2}(a\succ_{z_1}c+a\prec_{z_1}c)) +(z_1-z_2)^N b\succ_{z_2} (a\succ_{z_1}c) \\
&= (z_1-z_2)^N Y(b,z_2)Y(a,z_1)c.
\end{align*}
Thus, $(V,Y)$ satisfies the weak commutativity \meqref{2.1} and weak associativity \meqref{2.2}. Then by Theorem~\mref{thm2.2}, $(V,Y)$ is a vertex Leibniz algebra.

Finally, if there exists a linear map $D:V\ra V$ such that $(V, \prec_z, \succ_z,D )$ is a dendriform vertex algebra, then by \meqref{4.38} and  \meqref{4.39}, we have
$$e^{zD}Y(a,-z)b=e^{zD}(a\prec_{-z} b)+e^{zD}(a\succ_{-z} b)=b\succ_{z}a+b\prec_{z}a=Y(b,z)a,\quad a,b\in V.$$
Hence $(V,Y,D)$ satisfies the skew-symmetry \meqref{2.5}. Moreover,
\begin{align*}
D(Y(a,z)b)-Y(a,z)Db&= D(a\prec_z b)+D(a\succ_z b)-a\prec_z (Db)-a\succ_z (Db)\\
&=\frac{d}{dz}(a\prec_z Db)+\frac{d}{dz} (a\succ_z b)=\frac{d}{dz} Y(a,z)b, \quad a, b\in V.
\end{align*}
Hence $(V,Y,D)$ satisfies the $D$-bracket derivative property \meqref{2.4}. Then by Proposition \mref{prop2.6}, $Y$ also satisfies the weak commutativity, and by Theorem \mref{thm2.2}, $(V,Y,D)$ satisfies the Jacobi identity. This shows that $(V,Y,D)$ is a vertex algebra without vacuum.
\end{proof}

\subsection{Characterizations of the dendriform vertex (Leibniz) algebras}

Theorem \mref{thm4.5} shows that the equations \meqref{5.6}--\meqref{5.8} are the axioms underlying the weak associativity \meqref{2.2} of the vertex operator $Y$ given by $Y(a,z)b=a\prec_z b+a\succ_z b$, while equations \meqref{4.43} and \meqref{4.44} are the axioms underlying the weak commutativity \meqref{2.1}. On the other hand, by Proposition~\mref{prop2.6}, for any vertex operator map $Y$ on a vector space $V$, the weak associativity and weak commutativity are equivalent if there exists $D:V\ra V$ that satisfies some nice properties. So it is natural to expect the same kind of equivalency for the dendriform structures as well. In fact, we have the following conclusion.

\begin{prop}\mlabel{prop4.6}
Consider a quadruple $(V,\prec_z,\succ_z,D)$, where $V$ is a vector space, $\prec_z,\succ_z:V\ra \Hom(V,V((z)))$ are two linear maps, and $D:V\ra V$ is a linear map. Assume that $(V,\prec_z,\succ_z,D) $ satisfies \meqref{4.38} and \meqref{4.39}.
Then $(V,\prec_z,\succ_z,D) $ is a dendriform vertex algebra if and only if it satisfies \meqref{4.43} and \meqref{4.44}. In other words, in this case, equations \meqref{5.6}--\meqref{5.8} are equivalent to equations \meqref{4.43} and \meqref{4.44}.
    \end{prop}
    \begin{proof}
It follows from \meqref{4.39} that
        \begin{equation}\mlabel{4.45}
            e^{z_0D} a\prec_z  e^{-z_0 D} b=a\prec_{z+z_0} b,\quad e^{z_0D} a\succ_z  e^{-z_0 D} b=a\succ_{z+z_0} b, \quad a, b\in V.
        \end{equation}
The proof of \meqref{4.45} is similar to the proof of the conjugation formula of the vertex operator $e^{z_0D} Y(a,z)  e^{-z_0 D} =Y(a,z+z_0)$ by applying the $D$-bracket derivative property \meqref{2.4}, for which we refer the reader to \mcite{FHL,LL} for details.

By \meqref{4.45} and \meqref{4.38}, we can express the two sides of \meqref{5.6} as
        \begin{align*}
            (z_0+z_2)^{N}a\prec_{z_0+z_2} (b\succ_{z_2}c+b\prec_{z_2}c)&= (z_0+z_2)^N e^{z_2D}a\prec_{z_0}e^{-z_2 D} (b\succ_{z_2}c+b\prec_{z_2}c)\\
            &=(z_0+z_2)^N e^{z_2D}a\prec_{z_0}(c\prec_{-z_2}b+c\succ_{-z_2}b),
            %\numberthis\mlabel{4.46}
            \\
            (z_0+z_2)^{N}   (a\prec_{z_0}b)\prec_{z_2}c&=(z_0+z_2)^Ne^{z_2D} c\succ_{-z_2}(a\prec_{z_0}b),
        \end{align*}
        where $N\in \N$ depends on $a$ and $c$. Hence $(z_0+z_2)^N a\prec_{z_0}(c\prec_{-z_2}b+c\succ_{-z_2}b)=(z_0+z_2)^N c\succ_{-z_2}(a\prec_{z_0}b).$ By replacing $(z_0,z_2)$ with $(z_2,-z_1)$, and replacing $(c,a,b)$ with the ordered triple $(a,b,c)$ in this equation, we obtain
$$(z_2-z_1)^N b\prec_{z_2}(a\prec_{z_1}c+a\succ_{z_1}c)=(z_2-z_1)^N a\succ_{z_1}(b\prec_{z_2}c),$$
where $N$ depends on $a$ and $b$. This equation is equivalent to \meqref{4.43} since $N\geq 0$. Similarly, we can express the two sides of \meqref{5.7} as
        \begin{align*}
(z_0+z_2)^{N}a\succ_{z_0+z_2}(b\prec_{z_2}c)&=(z_0+z_2)^N e^{z_2 D} a\succ_{z_0}e^{-z_2D} (b\prec_{z_2}c)\\
            &=(z_0+z_2)^N e^{z_2 D} a\succ_{z_0}(c\succ_{-z_2}b),
            %\numberthis\mlabel{4.47}
            \\
            (z_0+z_2)^{N}   (a\succ _{z_0}b)\prec_{z_{2}}c&=    (z_0+z_2)^{N} e^{z_2D} c\succ_{-z_2}    (a\succ _{z_0}b),
        \end{align*}
where $N$ depends on $a$ and $c$.
Then $(z_0+z_2)^N  a\succ_{z_0}(c\succ_{-z_2}b)=(z_0+z_2)^{N}  c\succ_{-z_2}    (a\succ _{z_0}b),$ and by replacing $(z_0,z_2)$ with $(z_1,-z_2)$, and $(a,c,b)$ with $(a,b,c)$, we have $N$ depends on $a$ and $b$, and
$$(z_1-z_2)^N a\succ_{z_1}(b\succ_{z_2}c)=(z_1-z_2)^N b\succ_{z_2}(a\succ_{z_1}c),$$
which is \meqref{4.44}. Finally, we can express each side of \meqref{5.8} as
        \begin{align*}
(z_0+z_2)^{N}a\succ_{z_0+z_2}(b\succ_{z_2}c)&=  (z_0+z_2)^{N}e^{z_2D}a\succ_{z_0} e^{-z_2D}(b\succ_{z_2}c)\\
            &=(z_0+z_2)^{N}e^{z_2D}a\succ_{z_0} (c\prec_{-z_2}b),
            \\
            (z_0+z_2)^{N}   (a\succ_{z_0}b+a\prec_{z_0}b)\succ_{z_2} c&=    (z_0+z_2)^{N}   e^{z_2D}c\prec_{-z_2}(a\succ_{z_0}b+a\prec_{z_0}b),
        \end{align*}
        where $N$ depends on $a$ and $c$. Then $(z_0+z_2)^{N}a\succ_{z_0} (c\prec_{-z_2}b)= (z_0+z_2)^{N}   c\prec_{-z_2}(a\succ_{z_0}b+a\prec_{z_0}b),$ and by replacing $(z_0,z_2)$ with $(z_1,-z_2)$, and $(a,c,b)$ with $(a,b,c)$, we have $N$ depends on $a$ and $b$, and
        $$(z_1-z_2)^N a\succ_{z_1} (b\prec_{z_2}c)=(z_1-z_2)^N b\prec_{z_2}(a\succ_{z_1}c+a\prec_{z_1}c),$$
        which is \meqref{4.43}. Conversely, if $(V,\prec_z,\succ_z,D) $ satisfies \meqref{4.43} and \meqref{4.44}, by reversing the argument above, we can show that $(V,\prec_z,\succ_z,D) $ also satisfies \meqref{5.6}--\meqref{5.8}.
    \end{proof}
By Proposition~\mref{prop4.6}, any dendriform vertex algebra $(V,\prec_z,\succ_z,D) $ also satisfies \meqref{4.43} and \meqref{4.44}. Hence we have the following consequence.
\begin{coro}\mlabel{cor4.7}
Let $(V, \prec_z, \succ_z,D)$ be a dendriform vertex algebra. Then $(V,\prec_z,\succ_z)$ is a dendriform vertex Leibniz algebra.
\end{coro}
Therefore, we have the following diagram that illustrate the relations between our notions of dendriform algebras in Definition~\mref{df4.2} and the original (vertex) algebras:
$$
\begin{tikzcd}
\text{vertex\ alg.\ without\ vacuum} \arrow[r,hook,"\text{subcat.}"]& \text{vertex\ Leibniz\ alg.} \arrow [r,hook,"\text{full\ subcat.}"]& \text{field\ Leibniz\ alg.} \\
\text{dendriform\ vertex\ alg.} \arrow[u,"\text{induce}"]\arrow[r,hook,"\text{subcat.}"]&\text{dendriform\ vertex\ Leibniz\ alg.}\arrow[r,hook,"\text{full\ subcat.}"]\arrow[u,"\text{induce}"]& \text{dendriform\ field\ alg.}\arrow[u,"\text{induce}"]
\end{tikzcd}
$$

We can also obtain an analog of the Jacobi identity for the operators $\prec_z$ and $\succ_z$. We recall the following Lemma 2.1 in \mcite{L4}.
    \begin{lm}\mlabel{lm4.9}
        Let $U$ be a vector space, and let $A(z_1,z_2)\in U((z_1))((z_2))$, $B(z_1,z_2)\in U((z_2))((z_1))$, and $C(z_0,z_2)\in U((z_2))((z_0))$. Then
\begin{equation}\mlabel{4.49}
            z^{-1}_0\delta\left(\frac{z_1-z_2}{z_0}\right)A(z_1,z_2)-   z^{-1}_0\delta\left(\frac{-z_2+z_1}{z_0}\right) B(z_1,z_2)= z^{-1}_2\delta\left(\frac{z_1-z_0}{z_2}\right) C(z_0,z_2)
\end{equation}
        holds if and only if there exists $k,l\in \N$ such that
        \begin{align}
            (z_1-z_2)^k A(z_1,z_2)&=(z_1-z_2)^k B(z_1,z_2), \mlabel{4.50}\\
            (z_0+z_2)^l A(z_0+z_2,z_2)&=(z_0+z_2)^l C(z_0,z_2).\mlabel{4.51}
        \end{align}
    \end{lm}

\begin{thm}\mlabel{thm4.10}
        Let $(V,\prec_z,\succ_z)$ be a dendriform vertex Leibniz algebra. Then we have the following three Jacobi identities involving the operators $\prec_z$ and $\succ_z$.
        \begin{equation}\mlabel{4.52}
        \begin{aligned}
            &z^{-1}_0\delta\left(\frac{z_1-z_2}{z_0}\right)a\succ_{z_1}(b\prec_{z_2}c)- z^{-1}_0\delta\left(\frac{-z_2+z_1}{z_0}\right) b\prec_{z_2}(a\succ_{z_1}c+a\prec_{z_1}c)\\
            &=z^{-1}_2\delta\left(\frac{z_1-z_0}{z_2}\right) (a\succ_{z_0}b)\prec_{z_2}c,
            \end{aligned}
            \end{equation}
            \begin{equation}\mlabel{4.53}
            \begin{aligned}
            &z^{-1}_0\delta\left(\frac{z_1-z_2}{z_0}\right)a\succ_{z_1}(b\succ_{z_2}c)- z^{-1}_0\delta\left(\frac{-z_2+z_1}{z_0}\right) b\succ_{z_2}(a\succ_{z_1}c)\\
            &=z^{-1}_2\delta\left(\frac{z_1-z_0}{z_2}\right) (a\succ_{z_0}b+a\prec_{z_0}b)\succ_{z_2}c,
            \end{aligned}
            \end{equation}
            \begin{equation}\mlabel{4.54}
            \begin{aligned}
            &z^{-1}_0\delta\left(\frac{z_1-z_2}{z_0}\right)a\prec_{z_1}(b\prec_{z_2}c+b\succ_{z_2}c)-   z^{-1}_0\delta\left(\frac{-z_2+z_1}{z_0}\right) b\succ_{z_2}(a\prec_{z_1}c)\\
            &=z^{-1}_2\delta\left(\frac{z_1-z_0}{z_2}\right) (a\prec_{z_0}b)\prec_{z_2}c,
        \end{aligned}
        \end{equation}
where $a,b,c\in V$, and $z_0,z_1,z_2$ are formal variables.

Furthermore, \meqref{4.52}, \meqref{4.53} and \meqref{4.54} for a dendriform vertex algebra $(V,\prec_z,\succ_z,D)$ are mutually equivalent. We call \meqref{4.53} the Jacobi identity for the dendriform vertex algebra $(V,\prec_z,\succ_z,D)$.
    \end{thm}
    \begin{proof}
By Proposition \mref{prop4.6} and the formulas \meqref{5.6}-\meqref{5.8}, we have
        \begin{align*}
            (z_0+z_2)^k a\succ_{z_0+z_2}(b\prec_{z_2}c)&=(z_0+z_2)^k (a\succ_{z_0}b)\prec_{z_2}c,\\
            (z_1-z_2)^la\succ_{z_1}(b\prec_{z_2}c)&=(z_1-z_2)^l b\prec_{z_2}(a\succ_{z_1}c+a\prec_{z_1}c),
        \end{align*}
for some $k,l\in \N$. Then $A(z_1,z_2)=a\succ_{z_1}(b\prec_{z_2}c)$, $B(z_1,z_2)=b\prec_{z_2}(a\succ_{z_1}c+a\prec_{z_1}c)$, and $C(z_0,z_2)=(a\succ_{z_0}b)\prec_{z_2}c$ satisfy the conditions \meqref{4.50} and \meqref{4.51} in Lemma \mref{lm4.9}. Then the Jacobi identity \meqref{4.52} follows from \meqref{4.49}.

Similarly, the Jacobi identity \meqref{4.53} follows from Lemma \mref{lm4.9} and
        \begin{align*}
            (z_0+z_2)^k a\succ_{z_0+z_2}(b\succ_{z_2}c)&=(z_0+z_2)^k (a\succ_{z_0}b+a\prec_{z_0}b)\succ_{z_2}c,\\
            (z_1-z_2)^l a\succ_{z_1}(b\succ_{z_2}c)&=(z_1-z_2)^l b\succ_{z_2}(a\succ_{z_1}c),
        \end{align*}
for some $k,l\in \N$.
        The Jacobi identity \meqref{4.54} follows from Lemma \mref{lm4.9} and
        \begin{align*}
            (z_0+z_2)^k a\prec_{z_0+z_2}(b\succ_{z_2}c+b\prec_{z_2}c)&=(z_0+z_2)^k (a\prec_{z_0}b)\prec_{z_2}c,\\
            (z_1-z_2)^l a\prec_{z_1}(b\prec_{z_2}c+b\succ_{z_2}c)&=(z_1-z_2)^l b\succ_{z_2}(a\prec_{z_1}c),
        \end{align*}
for some $k,l\in \N$.

Now let $(V,\prec_z,\succ_z)$ be a dendriform vertex algebra. The equivalency of these Jacobi identities essentially corresponds to the $S_3$-symmetry of the Jacobi identity, see Section 2.7 in \mcite{FHL}.
The proof is also similar as follows.

Assume that \meqref{4.52} holds. By the skew-symmetry \meqref{4.38}, we have
{\small
        \begin{align*}
            &z^{-1}_0\delta\left(\frac{z_1-z_2}{z_0}\right)a\succ_{z_1}e^{z_2D}(c\succ_{-z_2}b)-    z^{-1}_0\delta\left(\frac{-z_2+z_1}{z_0}\right) e^{z_2D}(a\succ_{z_1}c+a\prec_{z_1}c)\succ_{-z_2}b\\
                &=z^{-1}_2\delta\left(\frac{z_1-z_0}{z_2}\right) (a\succ_{z_0}b)\prec_{z_2}c=z^{-1}_2\delta\left(\frac{z_1-z_0}{z_2}\right) e^{z_2D}c\succ_{-z_2}(a\succ_{z_0}b).
        \end{align*}
}
Then by \meqref{4.45} and properties of the formal $\delta$-functions (see Section 2.1 in \mcite{FHL}), we have
{\small
        \begin{align*}
        &z^{-1}_1\delta\left(\frac{z_2+z_0}{z_1}\right) c\succ_{-z_2}(a\succ_{z_0}b)=z^{-1}_2\delta\left(\frac{z_1-z_0}{z_2}\right) c\succ_{-z_2}(a\succ_{z_0}b)\\
        &=z^{-1}_0\delta\left(\frac{z_1-z_2}{z_0}\right)a\succ_{z_1-z_2}(c\succ_{-z_2}b)-   z^{-1}_0\delta\left(\frac{-z_2+z_1}{z_0}\right) (a\succ_{z_1}c+a\prec_{z_1}c)\succ_{-z_2}b\\
    &=  z^{-1}_1\delta\left(\frac{z_0+z_2}{z_1}\right)a\succ_{z_0}(c\succ_{-z_2}b)- z^{-1}_0\delta\left(\frac{-z_2+z_1}{z_0}\right) (a\succ_{z_1}c+a\prec_{z_1}c)\succ_{-z_2}b.
        \end{align*}
}
Changing the formal variables $(z_0,z_1,z_2)\mapsto (w_1,w_0,-w_2)$ in the equations above, we obtain
{\small
        \begin{align*}
        &w^{-1}_0\delta\left(\frac{-w_2+w_1}{w_0}\right) c\succ_{w_2}(a\succ_{w_1}b)\\
        &=w^{-1}_0\delta\left(\frac{w_1-w_2}{w_0}\right)a\succ_{w_1}(c\succ_{w_2}b)-    w^{-1}_1\delta\left(\frac{w_2+w_0}{w_1}\right) (a\succ_{w_0}c+a\prec_{w_0}c)\succ_{w_2}b.
        \end{align*}
}
This equation is the same as \meqref{4.53} under the change of variables $(c,b)$ to $(b,c)$. This shows the equivalence of \meqref{4.52} and \meqref{4.53}. The equivalency of the equations \meqref{4.54} and \meqref{4.53} can be proved by a similar method. Thus \meqref{4.52},  \meqref{4.53} and \meqref{4.54} are mutually equivalent.
\delete{ Then
{\small
        \begin{align*}
        &z^{-1}_2\delta\left(\frac{z_1-z_0}{z_2}\right) (a\prec_{z_0}b)\prec_{z_2}c=z^{-1}_1\delta\left(\frac{z_2+z_0}{z_1}\right) e^{z_2D}e^{z_0D}e^{-z_0D}c\succ_{-z_2}(e^{z_0D}b\succ_{-z_0}a)\\
        &=z^{-1}_1\delta\left(\frac{z_2+z_0}{z_1}\right)e^{z_1D}c\succ_{-z_1}(b\succ_{-z_0}a)\\
    &=  z^{-1}_0\delta\left(\frac{z_1-z_2}{z_0}\right)e^{z_1D}(b\prec_{z_2}c+b\succ_{z_2}c)\succ_{-z_1}a-   z^{-1}_0\delta\left(\frac{-z_2+z_1}{z_0}\right) b\succ_{z_2}e^{z_1D}(c\succ_{-z_1}a).
        \end{align*}
}
Hence by the properties of the formal $\delta$-function, we obtain
{\small     \begin{align*}
    &   z^{-1}_2\delta\left(\frac{z_1-z_0}{z_2}\right)c\succ_{-z_1}(b\succ_{-z_0}a)=z^{-1}_1\delta\left(\frac{z_2+z_0}{z_1}\right)c\succ_{-z_1}(b\succ_{-z_0}a)\\
    &=z^{-1}_0\delta\left(\frac{z_1-z_2}{z_0}\right)(b\prec_{z_2}c+b\succ_{z_2}c)\succ_{-z_1}a- z^{-1}_0\delta\left(\frac{-z_2+z_1}{z_0}\right) b\succ_{z_2-z_1}(c\succ_{-z_1}a)\\
    &=z^{-1}_1\delta\left(\frac{z_0+z_2}{z_1}\right)(b\prec_{z_2}c+b\succ_{z_2}c)\succ_{-z_1}a+ z^{-1}_2\delta\left(\frac{-z_0+z_1}{z_2}\right) b\succ_{-z_0}(c\succ_{-z_1}a).
        \end{align*}
}
Under the change of the variables $(z_0,z_1,z_2)\mapsto (-w_1,-w_2,w_0)$, the above equation becomes
{\small     \begin{align*}
        &w^{-1}_0\delta\left(\frac{-w_2+w_1}{w_0}\right)c\succ_{w_2}(b\succ_{w_1}a)\\
        &=-w^{-1}_2\delta\left(\frac{w_1-w_0}{w_2}\right)(b\prec_{w_2}c+b\succ_{w_2}c)\succ_{w_2}a+ w^{-1}_0\delta\left(\frac{w_1-w_2}{w_0}\right) b\succ_{w_1}(c\succ_{w_2}a).
        \end{align*}
    }
It is obvious that this equation is the same as \meqref{4.53}. Hence the equations \meqref{4.54} and \meqref{4.53} are equivalent.} 
    \end{proof}
    \begin{remark}
By adding the three Jacobi identities \meqref{4.52}--\meqref{4.54}, we derive the Jacobi identity for the vertex operator $Y(a,z)=a\prec_z b+a\succ_z b$, giving an alternative proof of Theorem \mref{thm4.5}
    \end{remark}
On the other hand, since the Jacobi identity \meqref{4.49} can also give rise to \meqref{4.50} and \meqref{4.51}, we have the following characterization of a dendriform vertex Leibniz algebra:

\begin{coro}\mlabel{coro4.13}
    A dendriform vertex Leibniz algebra is a vector space $V$, equipped with two linear operators $\prec_z,\succ_{z}:V\ra \Hom(V,V((z)))$, satisfying the Jacobi identities \meqref{4.52}, \meqref{4.53}, and \meqref{4.54}.
    \end{coro}

By Theorem~\mref{thm4.10}, we also obtain a second equivalent condition for dendriform vertex algebras, in addition to Proposition~\mref{prop4.6}.
\begin{coro}\mlabel{coro4.14}
        A dendriform vertex algebra is a vector space $V$, equipped with a linear map $D:V\ra V$ and two linear operators $\prec_z,\succ_z: V\ra \Hom(V,V((z)))$, satisfying \meqref{4.38}, \meqref{4.39}, and the Jacobi identity \meqref{4.53}.
    \end{coro}

\subsection{The modules structures induced by dendriform vertex (Leibniz) algebras}

An usual dendriform associative algebra $(A,\prec,\succ)$ defined
by \meqref{4.30}--\meqref{4.32} gives rise to a bi-module structure $(A,L_\succ,R_\prec)$ of $(A,\cdot)$ on
$A$ itself, where $a\cdot b:=a\prec b+a\succ b$,
$L_{\succ}(a)(b):=a\succ b$, and $R_{\prec}(a)(b):=a\prec b$, for
all $a,b\in A$. See \mcite{B3,BGN2} for more details.

It is natural to expect a similar result to be true for our definition of the dendriform vertex (Leibniz) algebras in Definition~\ref{df4.3}. However, unlike the associative algebra case, vertex algebra also satisfies the weak commutativity and a module over a vertex algebra is more like a module over Lie algebras. In fact, with Theorem~\mref{thm4.10}, we indeed have a module structure induced by a dendriform vertex algebra. First, we recall the following definition. See Definition 2.9 in \mcite{LTW}.

\begin{df}\mlabel{df4.14}
    Let $(V,Y,D)$ be a vertex algebra without vacuum. A \name{$V$-module} $(W,Y_W)$ is a vector space $W$, equipped with a linear map $Y_W:V\ra (\End W)[[z,z^{-1}]]$, satisfying the truncation property, the Jacobi identity for $Y_W$ in Definition \mref{df2.13}, and
\begin{equation}
Y_W(Da,z)=\frac{d}{dz}Y_W(a,z),\quad a\in V.
\end{equation}
    \end{df}
\begin{prop}\mlabel{prop4.15}
Let $(V,\prec_z,\succ_z,D)$ be a dendriform vertex algebra, and let $(V,Y,D)$ be the associated vertex algebra without vacuum, where $Y$ is given by \meqref{5.11}: $Y(a,z)b=a\prec_z b+a\succ_zb$. Let $W=V$, and define
    \begin{equation}\mlabel{4.58'}
    Y_W: V\ra (\End W)[[z,z^{-1}]],\quad Y_W(a,z)b:=a\succ_{z}b, \quad a\in V, b\in W.
    \end{equation}
Then $(W,Y_W)$ is a module over $(V,Y,D)$.
    \end{prop}
\begin{proof}
By Definition \mref{df4.3}, clearly $Y_W$ satisfies the truncation
property. By the Jacobi identity \meqref{4.53} of
$(V,\prec_z,\succ_z,D)$, \meqref{4.58'} and
\meqref{5.11}, we have {\small
        \begin{align*}
        &z^{-1}_0\delta\left(\frac{z_1-z_2}{z_0}\right)Y_W(a,z_1)Y_W(b,z_2)c    -z^{-1}_0\delta\left(\frac{-z_2+z_1}{z_0}\right) Y_W(b,z_2)Y_W(a,z_1)c
        =z^{-1}_2\delta\left(\frac{z_1-z_0}{z_2}\right) Y_W(Y(a,z_0)b,z_2)c,
    \end{align*}
}
for $a, b\in V, c\in W=V$. Finally, by Lemma \mref{lm4.4}, we have $$Y_{W}(Da,z)b=(Da)\succ_z b=\frac{d}{dz}a\succ_zb, \quad a\in V, b\in W.$$
Thus, $(W,Y_W)$ is a module over the vertex algebra without vacuum $(V,Y,D)$, according to Definition \mref{df4.14}.
    \end{proof}
\begin{remark}\mlabel{rk4.16}
Since we define $Y_W$ by one of the partial operator $\succ_z$ in \meqref{4.58'}, it is natural to consider the vertex operator $\Y$ defined by the other partial operator $\Y(a,z)b=a\prec_z b$. By the skew-symmetry \meqref{4.38}, we have
        \begin{equation}\mlabel{4.59'}
        \Y(a,z)b=a\prec_z b=e^{zD}(b\succ_{-z}a)=e^{zD}Y_W(b,-z)a=Y_{WV}^W(a,z)b,   \end{equation}
where $Y_{WV}^W$ is defined by the skew-symmetry formula (see Section 5 in \mcite{FHL}). 
        i.e., $\Y=Y_{WV}^W$. It is easy to see that the Jacobi identities \meqref{4.52} and \meqref{4.54} correspond to the following equation.
        \begin{align*}z
        &z^{-1}_0\delta\left(\frac{z_1-z_2}{z_0}\right)Y_W(a,z_1)Y_{WV}^W(b,z_2)c-  z^{-1}_0\delta\left(\frac{-z_2+z_1}{z_0}\right) Y_{WV}^W(b,z_2)Y(a,z_1)c\\
        &=z^{-1}_2\delta\left(\frac{z_1-z_0}{z_2}\right) Y_{WV}^{W}(Y_{W}(a,z_0)b,z_2)c,\quad a,b,c\in V=W.
    \end{align*}
Moreover, $Y_{WV}^W(Da,z)b=(Da)\prec_zb=\frac{d}{dz}a\prec_zb$ by Lemma \mref{lm4.4}. Thus, if the vertex algebra without vacuum $(V,Y,D)$ is an underlying structure of some VOA $(V,Y,\vac,\om)$, with $D=L(-1)$, then $Y_{WV}^W(a,z)b=a\prec_zb$ is an intertwining operator of type $\fusion{W}{V}{W}$.
    \end{remark}
\begin{coro} \mlabel{co:denrb}
    Let $(V,Y,\vac,\om)$ be a VOA. Assume that the underlying vertex algebra without vacuum structure $(V,Y,D=L(-1))$ is induced from a dendriform vertex algebra structure $(V,\prec_z,\succ_z,D)$ by Theorem \mref{thm4.5}. Let $(W,Y_W)$ be the weak $V$-module given by Proposition \mref{prop4.15}. Then the identity map $T=\Id: W\ra V$ is a relative RBO.
    \end{coro}
\begin{proof}
    By \meqref{5.11}, \meqref{4.58'}, \meqref{4.59'}, and the assumption that $T=\Id$, we have
\begin{align*}
Y(Tu,z)Tv&=u\prec_z v+u\succ_z v=Y_W(u,z)v+Y_{WV}^W(u,z)v\\
&=T(Y_W(Tu,z)v)+T(Y_{WV}^W(u,z)Tv),\quad u,v\in W.
\end{align*}
So $T=\Id$ is a relative RBO as defined in \meqref{4.60'}
    \end{proof}
Conversely, we have the following relation between dendriform vertex algebras and vertex algebra without vacuum, which gives a characterization of the dendriform vertex algebras.

\begin{prop}\mlabel{prop4.19}
    Let $V$ be a vector space, equipped with linear maps $\prec_z,\succ_{z}: V\ra \Hom(V,V((z)))$, and $D:V\ra V$. Define $Y(a,z)b:=a\prec_zb+a\succ_zb$ for all $a,b\in V$ as in \meqref{5.11}. Then $(V,\prec_z,\succ_z,D)$ is a dendriform vertex algebra if and only if the following conditions are satisfied.
    \begin{enumerate}
    \item $(V,Y,D)$ is a vertex algebra without vacuum.
    \item $\prec_z$, $\succ_z$, and $D$ satisfy the equations \meqref{4.38} and \meqref{4.39}.
    \item $(V, \succ_z, D)$ defines a module structure of $(V,Y,D)$ on $V$ itself.
    \end{enumerate}
    \end{prop}
\begin{proof}
If $(V,\prec_z,\succ_z, D)$ forms a dendriform vertex algebra, then (i), (ii), and (iii) follows from Theorem ~\mref{thm4.5}, Definition~\mref{df4.3}, and Proposition~\mref{prop4.15}, respectively. Conversely, since $\succ_{z}:V\ra \End(V)[[z,z^{-1}]]$ defines a module structure, we have the Jacobi identity:
{\small     \begin{align*}
&z^{-1}_0\delta\left(\frac{z_1-z_2}{z_0}\right)a\succ_{z_1}(b\succ_{z_2}c)- z^{-1}_0\delta\left(\frac{-z_2+z_1}{z_0}\right) b\succ_{z_2}(a\succ_{z_1}c)\\
&=z^{-1}_2\delta\left(\frac{z_1-z_0}{z_2}\right) (Y(a,z_0)b)\prec_{z_2}c\\
&=z^{-1}_2\delta\left(\frac{z_1-z_0}{z_2}\right) (a\succ_{z_0}b+a\prec_{z_0}b)\prec_{z_2}c,\quad a,b,c\in V.
\end{align*}
} Hence $(V,\prec_z,\succ_z,D)$ is a quadruple satisfying
\meqref{4.38}, \meqref{4.39}, and \meqref{4.53}. Then by
Corollary~\mref{coro4.14},  $(V,\prec_z,\succ_z,D)$ is a
dendriform field vertex algebra.
\end{proof}

However, for dendriform vertex Leibniz algebras, since we do not have the skew-symmetry property, the left and right module actions cannot be combined into one action. The notion of left module over a vertex Leibniz algebra was introduced in \mcite{LTW} (see Definition 2.15 there). Inspired by the Jacobi identities in Theorem~\mref{thm4.10} and Proposition~\mref{prop4.15}, we introduce the notion of bi-modules over a vertex Leibniz algebra as follows.

\begin{df}\mlabel{df4.21}
    Let $(V,Y)$ be a vertex Leibniz algebra. A {\bf bi-module over $(V,Y)$} is a triple $(W,Y_W,Y_{WV}^W)$, where $W$ is a vector space, and
    $$Y_W:V\ra \Hom(W,W((z))),\quad Y_{WV}^W: W\ra \Hom(V,W((z)))$$ are linear operators, satisfying the following axioms:
{\small \begin{equation}
\begin{split}   &z^{-1}_0\delta\left(\frac{z_1-z_2}{z_0}\right)Y_W(a,z_1)Y_{WV}^W(b,z_2)c-  z^{-1}_0\delta\left(\frac{-z_2+z_1}{z_0}\right) Y_{WV}^W(b,z_2)Y(a,z_1)c\\
    &=z^{-1}_2\delta\left(\frac{z_1-z_0}{z_2}\right) Y_{WV}^W(Y_W(a,z_0)b,z_2)c,\quad  a,c\in V, b\in W,
\end{split}
\mlabel{4.63}
\end{equation}
\begin{equation}
\begin{split}
&z^{-1}_0\delta\left(\frac{z_1-z_2}{z_0}\right)Y_W(a,z_1)Y_W(b,z_2)c-   z^{-1}_0\delta\left(\frac{-z_2+z_1}{z_0}\right) Y_W(b,z_2)Y_W(a,z_1)c\\
&=z^{-1}_2\delta\left(\frac{z_1-z_0}{z_2}\right) Y_W(Y(a,z_0)b,z_2)c,\quad  a,b\in V, c\in W,
\end{split}
\mlabel{4.64}
\end{equation}
\begin{equation}
\begin{split}
&z^{-1}_0\delta\left(\frac{z_1-z_2}{z_0}\right)Y_{WV}^W(a,z_1) Y(b,z_2)c-   z^{-1}_0\delta\left(\frac{-z_2+z_1}{z_0}\right) Y_W(b,z_2)Y_{WV}^W(a,z_1)c\\
    &=z^{-1}_2\delta\left(\frac{z_1-z_0}{z_2}\right) Y_{WV}^W(Y_{WV}^W(a,z_0)b,z_2)c, \quad  b,c\in V, a\in W.
\end{split}
\mlabel{4.65}
\end{equation}
}
    In particular, $(W,Y_W)$ is a (left) module over the vertex Leibniz algebra $(V,Y)$ as in \mcite{LTW}, in view of \meqref{4.64}.
\end{df}

\begin{prop}\mlabel{prop4.21}
    Let $(V,\prec_z,\succ_z)$ be a dendriform vertex Leibniz algebra, and let $(V,Y)$ be the associated vertex Leibniz algebra, where $Y$ is given by \meqref{5.11}: $Y(a,z)b=a\prec_z b+a\succ_z b$. Let $W=V$, and define $Y_W$ and $Y_{WV}^W$ by
    \begin{equation}\mlabel{4.66}
    Y_W(a,z)b:=a\succ_z b,\quad Y_{WV}^W(b,z)a:=b\prec_z a,\quad \forall a\in V, b\in W.
    \end{equation}
    Then $(W,Y_W,Y_{WV}^W)$ is a bi-module over $(V,Y)$.
    \end{prop}
\begin{proof}
We use the equivalent definition of dendriform vertex Leibniz algebra in Corollary~\mref{coro4.13}. By our definition \meqref{4.66}, it is clear that \meqref{4.63}, \meqref{4.64}, and \meqref{4.65} are equivalent to \meqref{4.52}, \meqref{4.53}, and \meqref{4.54}, respectively. Therefore,  $(W,Y_W,Y_{WV}^W)$ is a bi-module over $(V,Y)$ by Definition~\mref{df4.21}.
\end{proof}

Similar to Proposition~\mref{prop4.19}, we have the following characterization of the dendriform vertex Leibniz algebras.

\begin{prop}
    Let $V$ be a vector space, equipped with linear maps $\prec_z,\succ_z: V\ra \Hom(V,V((z)))$. Let $Y(a,z)b=a\prec_z b+a\succ_z b$ for all $a,b\in V$. Then $(V,\prec_z, \succ_z)$ forms a dendriform vertex Leibniz algebra if and only if
    \begin{enumerate}
        \item $(V,Y)$ forms a vertex Leibniz algebra, and
        \item $(V,\succ_z,\prec_z)$ forms a bi-module of the vertex Leibniz algebra $(V,Y)$.
        \end{enumerate}
    \end{prop}
\begin{proof}
    If $(V,\prec_z,\succ_z)$ forms a dendriform vertex Leibniz algebra, then (i) follows from Theorem~\mref{thm4.5}, (ii) follows from Proposition~\mref{prop4.21}. Conversely, by our assumption and Definition~\mref{df4.21}, equations \meqref{4.63}--\meqref{4.65} translates to equations \meqref{4.52}--\meqref{4.54}. Then by Corollary~\mref{coro4.13}, $(V,\prec_z,\succ_z)$ is a dendriform vertex Leibniz algebra.
    \end{proof}

\begin{remark}
    For dendriform field algebras, we can follow the same routine and introduce a notion of bi-modules over a field Leibniz algebra, which is coherent with equations \meqref{5.6}--\meqref{5.8}. Then the axioms of dendriform field algebras can also be characterized by the bi-module axioms.
\end{remark}

\noindent
{\bf Acknowledgments.}
This research is supported by
NSFC (11931009, 12271265, 12261131498), the Fundamental Research Funds for the Central Universities and Nankai Zhide Foundation.

\smallskip

\noindent
{\bf Declaration of interests. } The authors have no conflicts of interest to disclose.

\smallskip

\noindent
{\bf Data availability. } No new data were created or analyzed in this study.

\end{document}